\documentclass[review]{elsarticle}

\usepackage{lineno,hyperref}
\usepackage{dsfont}
\usepackage{tikz}
\usepackage{amsmath,amssymb,amsfonts,amsthm}

\usepackage{algorithm}
\usepackage{bm}
\usepackage{bookmark}

\newtheorem{thm}{Theorem}
\newtheorem{lem}[thm]{Lemma}
\newtheorem{proposition}{Proposition}

\theoremstyle{definition}
\newtheorem{dfn}{Definition}
\newdefinition{rmk}{Remark}

\theoremstyle{remark}

\journal{Journal of \LaTeX\ Templates}

\makeatletter
\def\ps@pprintTitle{%
   \let\@oddhead\@empty
   \let\@evenhead\@empty
   \def\@oddfoot{\reset@font\hfil\thepage\hfil}
   \let\@evenfoot\@oddfoot
}
\makeatother

\begin{document}

\begin{frontmatter}

\title{The constructions of directed strongly regular graph by Algebraic method}
\author[]{Yiqin He}
\ead{2014750113@smail.xtu.edu.cn}
\author[]{Bicheng Zhang\corref{cor1}}
\ead{zhangbicheng@xtu.edu.cn}
\author[]{Huabin Cao}
\ead{2014750117@smail.xtu.edu.cn}
\address{Department of Mathematics and Computational Science, Xiangtan Univerisity, Xiangtan, Hunan, 411105, PR China}
\cortext[cor1]{Corresponding author}



%

\begin{abstract}
The concept of directed strongly regular graphs (DSRG) was introduced by Duval in ``A Directed Graph Version of Strongly Regular Graphs" [Journal of Combinatorial Theory, Series A 47(1988)71-100]. Duval also provided several construction methods for directed strongly regular graphs. In this paper, We construct several new families of directed strongly regular graphs which are obtained by using Kronecker matrix product, Semidirect product, Cayley coset graph.\;At the same time, we give some sufficient and necessary conditions of two special families of Cayley graphs to be DSRG by using group representation theory. At last, we discuss some propositions of in(out)-neighbour set and automorphism group in directed strongly regular graphs.
\end{abstract}
\end{frontmatter}
\section{Introduction}

A \emph{directed strongly regular graph} (DSRG) with parameters $( n, k, \mu ,\lambda ,t)$ is a $k$-regular directed graph on $n$ vertices such that every vertex is on $t$ 2-cycles,\;and the number of paths of length two from a vertex $x$ to a vertex $y$ is $\lambda$ if there is an edge directed from $x$ to $y$ and it is $\mu$ otherwise. A DSRG with $t=k$ is an (undirected) strongly regular graph (SRG).\;Duval  showed that DSRGs with $t=0$ are the doubly regular tournaments. It is therefore usually assumed that $0<t<k$.

Another definition of a directed strongly regular graph, in terms of its adjacency matrix. Let $D$ be a directed graph
with $n$ vertices. Let $A=A(D)$ denote the adjacency matrix of $D$, and let $I = I_n$ and
$J = J_n$ denote the $n\times n$ identity matrix and all-ones matrix, respectively. Then $D$ is a directed strongly regular graph with parameters $(n,k,\mu,\lambda,t)$ if
and only if (i) $JA = AJ = kJ$ and (ii) $A^2=tI+\lambda A+\mu(J-I-A)$.\;Duval gave some propositions of DSRG in \cite{A} .

We shall assume that the reader is familiar with standard  terminology on directed graphs(see,e.g.,\cite{BO}).The vertex set and the arc set of a digraph $D$ are denote by $V(D)$ and $E(D)$. If $xy$ is an arc of digraph $D$, we say $x$ \emph{dominates} $y$ and write $x\rightarrow y$. Then the \emph{out-neighbour set} $N_D^+(x)$ of a vertex $x$ is the set of vertices dominated by $x$, and the \emph{in-neighbour set} $N_D^-(x)$ of a vertex $x$ is the set of vertices dominating $x$.
The number $d_D^+(v)=|N_D^+(x)|$ and $d_D^-(v)=|N_D^-(x)|$ are the \emph{outdegree} and \emph{indegree} of $x$.
\begin{proposition}(see \cite{A})\label{p-1}
A directed strongly regular graph with parameters $( n,k,\mu ,\lambda ,t)$ with $0<t<k$ satisfy
\[k(k+(\mu-\lambda ))=t+\left(n-1\right)\mu,\]
\[{{d}^{2}}={{\left( \mu -\lambda  \right)}^{2}}+4\left(t-\mu\right)\text{,}d|2k-(\lambda-\mu)(n-1),\]
\[\frac{2k-(\lambda-\mu)(n-1)}{d}\equiv n-1(mod\hspace{2pt}2)\text{,}\left|\frac{2k-(\lambda-\mu)(n-1)}{d}\right|\leq n-1,\]
where d is a positive integer, and
$$ 0\le \lambda <t,0<\mu \le t,-2\left( k-t-1 \right)\le \mu -\lambda \le 2\left( k-t \right).$$
\end{proposition}
\begin{proposition}(see \cite{A})\label{p-2}
A directed strongly regular graph with parameters $( n,k,\mu ,\lambda ,t)$ has three distinct integer eigenvalues
$$ k>\rho =\frac{1}{2}\left( -\left( \mu -\lambda  \right)+d \right)>\sigma =\frac{1}{2}\left( -\left( \mu -\lambda  \right)-d \right), $$
The multiplicities are
$$  1,r=-\frac{k+\sigma \left( n-1 \right)}{\rho -\sigma },s=\frac{k+\rho \left( n-1 \right)}{\rho -\sigma },~~~\eqno(1)$$
respectively.
\end{proposition}
\begin{proposition}(see \cite{A})\label{p-3}
If $G$ is a DSRG with parameters $( n,k,\mu ,\lambda ,t)$, then the complementary $G'$ is also a DSRG with parameters $(n',k',\mu',\lambda' ,t')$, where $k'=(n-2k)+(k-1)$, $\lambda'=(n-2k)+(\mu-2)$, $t'=(n-2k)+(t-1)$, $\mu'=(n-2k)+\lambda$.
\end{proposition}
We note that, although two DSRGs have the same parameters, they may not be isomorphic.

There are many constuction methods for DSRG. Duval described some methods including constructions using quadratic residue, block construction of permutation matrices and the Kronecker product \cite{A}. In addition, some of known constructions use combinatorial block designs \cite{FI}, coherent algebras\cite{FI}, semidirect product \cite{D}, finite geometries \cite{F,FI}, matrices\cite{K1}, block matrices\cite{AD}, finite incidence structures \cite{OL,BR} and Cayley graph\cite{HO,FI} or generalized Cayley graph \cite{FE}.

In section 2, we construct 3 families of DSRGs by using Kronecker matrix product.
We construct DSRGs with parameters $(4n^2,2n^2-2,n^2-1,n^2-3,n^2-1)$, $(4n^2,2n^2,n^2+1,n^2-1,n^2+1)$ for odd integer $n$ and $(4n^2,2n^2,n^2+4,n^2-4,n^2+4)$ for integer $4|n$.

In section 3,\;by using \emph{semidirect product},\;we construct 4 families of DSRGs with parameters:\\ $(1)(pn,vn,\frac{n}{p-1}v^2,\frac{n}{p-1}v(v-1),\frac{n}{p-1}v^2)$ for $1\leq v\leq p-2$,\\
$(2)(pn,n(v+1)-1,\frac{n}{p-1}v(v+1),n-2+\frac{n}{p-1}v^2,n-1+\frac{n}{p-1}v^2)$ for $1\leq v\leq p-2$;\\ $(3)(mq,m+q-2,\frac{m-1}{q}+1,\frac{m-1}{q}+q-2,\frac{m-1}{q}+q-1)$;\\
$(4)(p^{2}n, p(p-1)n, n((p-1)^2+1), n\frac{(p-1)^3-1}{p-1}, n((p-1)^2+1))$.

In section 4, we give a sufficient and necessary condition of Cayley coset graph to be DSRG with parameters $(n,k,\mu,\lambda,t)$ in terms of the group ring.

In section 5, we focus on the Cayley graph $\mathcal{C}(G,S)$ where the group $G$ is two semidirect product of two cyclic group. By using group representation theory, we can determine the eigenvalues and minimum polynomial of the adjacent matrix of Cayley graph, so we can give some sufficient and necessary condition of Cayley graphs $\mathcal{C}{(C_n\rtimes_{\theta(k)} C_m,A'\times C_m)}$ and $\mathcal{C}{(C_n\rtimes_{\theta(k)} C_m,(A'+e_A)\times C_m\setminus e_A)}$ to be DSRG,\;where $A'$ is a proper subset of $C_n\setminus e $ ($e$ is the identity element of group $C_n$). This suggests that the group representation theory can be used to investigate DSRG.

In section 6, we make a discussion of the vertices which have the same out-neighbour set (in-neighbours set) in DSRG. Meanwhile,\;we also give an upper bound of the order of automorphism group of directed strongly regular Cayley graph.
\section{Constructions of DSRG by using Kronecker matrix product}
In this section, we construct directed strongly regular graphs from known DSRGs in terms of Kronecker matrix
product. Duval \cite{A} observed that if $t=\mu$ and $A$  satisfies above equations (i) and (ii),
then so does $A\otimes J_m$ for every positive integer $m$; and so, we have:
\begin{proposition}{(see \cite{A})}\label{p-4}
If there exists a directed strongly regular graph with parameters $( n,k,\mu ,\lambda ,t)$ with $t=\mu$, then for each positive integer m there exists a directed strongly regular graph with parameters $( mn,mk,m\mu ,m\lambda ,mt)$
\end{proposition}
Duval \cite{A} also observed that if there exists a DSRG with parameters $( n,k,\mu ,\lambda ,t)$ with $t=\lambda+1$, then digraph corresponding to adjacency matrix $A\otimes J_m+I_n\otimes(J_m-I_m)$ is a DSRG  with parameters
$(mn,m{k+1}-1,m\mu,m(t+1)-2,m(t+1)-1)$.
\begin{dfn}\label{d-1}
Let $A=(a_{ij})_{m\times n}$ and $B=(b_{ij})_{t\times s}$ are two matrices over number field $\mathbb{K}$, then the Kronecker matrix
product of $A$ and $B$ is defined by
\[A\otimes B=
\left(\begin{array}{cccc}
a_{11}B & a_{12}B& \cdots&a_{1n}B\\
a_{21}B & a_{22}B& \cdots&a_{2n}B\\
\vdots& \vdots& \ddots&\vdots\\
a_{m1}B & a_{m2}B& \cdots&a_{mn}B\\
\end{array}\right)
.\]
\end{dfn}
Let $D(A)$ denote with the digraph respect to adjacent matrix $A$.
\begin{thm}\label{t-1}
Let $D$ be a directed strongly regular graph with the parameters $(n,k,\mu,\lambda,t)$ satisfies $t=\mu$ and  $4k=n+2\lambda+2\mu$. Let $A=A(D)$ be the adjacent matrix of $D$,and $B=(J-A)\otimes{A}+A\otimes(J-A)$. Then $D(B)$ is a directed strongly regular graph with the parameters
\[(n^2,2k(n-k),2(k^2-2\mu\lambda),2(k^2-\lambda^2-\mu^2),2(k^2-2\mu\lambda)).\]
\end{thm}
\begin{proof}
Note that $J^2=nJ$, and $JA=AJ=kJ$. Thus by a simple computation, we can obtain that
\[(J\otimes{J})B=J(J-A)\otimes{JA}+JA\otimes{J(J-A)}=2k(n-k)J\otimes{J},\]
\[B(J\otimes{J})=(J-A)J\otimes{AJ}+AJ\otimes{(J-A)J}=2k(n-k)J\otimes{J},\]
\[B^2=(J-A)^2\otimes{A^2}+A^2\otimes(J-A)^2+2(J-A)A\otimes{A(J-A)}.\tag{2}\]
Note that
\[A^2=tI+\lambda A+\mu(J-I-A)=\mu{J}+(\lambda-\mu)A,\]
\[(J-A)A=A(J-A)=(k-\mu){J}-(\lambda-\mu)A,\]
\[(J-A)^2=J^2+A^2-2kJ=(n-2k+\mu)J+(\lambda-\mu)A.\]
Thus, a simple calculation shows that
$$(J-A)^2\otimes{A^2}=((n-2k+\mu)J+(\lambda-\mu)A)\otimes(\mu{J}+(\lambda-\mu)A)$$
$$\hspace{71pt}=\mu(n-2k+\mu)J\otimes{J}+(n-2k+\mu)(\lambda-\mu)J\otimes{A}$$
\[\hspace{1pt}+\mu(\lambda-\mu)A\otimes{J}+(\lambda-\mu)^2{A\otimes{A}},\tag{3}\]
$${A^2}\otimes{(J-A)^2}=(\mu{J}+(\lambda-\mu)A)\otimes((n-2k+\mu)J+(\lambda-\mu)A)$$
$$\hspace{71pt}=\mu(n-2k+\mu)J\otimes{J}+(n-2k+\mu)(\lambda-\mu)A\otimes{J}$$
\[\hspace{1pt}+\mu(\lambda-\mu)J\otimes{A}+(\lambda-\mu)^2{A\otimes{A}},\tag{4}\]
and
$$2(J-A)A\otimes{A(J-A)}=2((k-\mu){J}-(\lambda-\mu)A)\otimes((k-\mu){J}-(\lambda-\mu)A)$$
\[=2(k-\mu)^2{J}\otimes{J}+2(\lambda-\mu)^2{A\otimes{A}}-2(\lambda-\mu)(k-\mu)(J\otimes{A}+A\otimes{J}).\tag{5}\]
We note that $2\mu(n-2k+\mu)+2(k-\mu)^2=2\mu(2k-2\lambda-\mu)+2(k-\mu)^2=2(k^2-2\mu\lambda)$. Thus from $(2),(3),(4),(5)$,
\[B^2=(2\mu(n-2k+\mu)+2(k-2\mu)^2)J\otimes{J}+4(\lambda-\mu)^2{A\otimes{A}}\hspace{10pt}\]
\[+(\lambda-\mu)(n-2k+2\mu-2k+2\mu)(J\otimes{A}+A\otimes{J})\]
\[\hspace{4pt}=2(k^2-2\mu\lambda)J\otimes{J}+(\lambda-\mu)(2n-8k+4\mu+4\lambda)A\otimes{A}\]
\[\hspace{7pt}+(\lambda-\mu)(n-4k+4\mu)((J-A)\otimes{A}+A\otimes(J-A))\]
\[=2(k^2-2\mu\lambda)J\otimes{J}+(\lambda-\mu)(n-4k+4\mu)B.\hspace{42pt}\]

  Hence $D(B)$ is a DSRG, let the parameters of $D(B)$ be $(n_1,k_1,\mu_1,\lambda_1,t_1)$, then $n_1=n^2,k_1=2k(n-k),\mu_1=t_1=2(k^2-2\mu\lambda),\lambda_1=2(k^2-2\mu\lambda)+(\lambda-\mu)(n-4k+4\mu)=2(k^2-2\mu\lambda)-2(\lambda-\mu)^2=2(k^2-\lambda^2-\mu^2)$.\;This completes the proof.
\end{proof}
\begin{lem}{(see \cite{JG},\cite{MA})}\label{l-2}
If $(n,k,\mu,\lambda,t)$ are the parameters of a DSRG with $t=\mu$ and rank $r$ and with $\frac{k}{n}=\frac{a}{b}$, where $a$ and $b$ are relatively prime integers, then
$$( n,k,\mu ,\lambda ,t)=( \frac{(r-1){{b}^{2}}}{c}m,\frac{(r-1)ab}{c}m,\frac{r{{a}^{2}}}{c}m,\frac{(ar-b)a}{c}m,\frac{r{{a}^{2}}}{c}m )$$
for some positive integer $m$, where $c$ is the greatest common divisor
$$c=\gcd (ab,r{{a}^{2}},(r-1){{b}^{2}}). $$
\end{lem}
We now give the possible parameters of DSRGs with $t=\mu$, and $4k=n+2\lambda+2\mu$.
\begin{thm}\label{t-3}
If $( n,k,\mu ,\lambda ,t)$ are the parameters of a DSRG with $t=\mu$, $4k=n+2\lambda+2\mu$, then for positive integer $m$, the possible parameters have just 4 classes: \\
(1)$(4(r-1)m,2(r-1)m,rm,(r-2)m,rm)$ for odd integer $r$;\\
(2)$(2(r-1)m,(r-1)m,\frac{r}{2}m,(\frac{r}{2}-1)m,\frac{r}{2}m)$ for even integer $r$;\\
(3)$(2rm,(r-1)m,\frac{r-1}{2}m,\frac{r-3}{2}m,\frac{r-1}{2}m)$ for odd integer $r$;\\
(4)$(4rm,2(r-1)m,(r-1)m,(r-3)m,(r-1)m)$ for even integer $r$.
\end{thm}
\begin{proof}
From Lemma \ref{l-2}, let $n=\frac{(r-1){{b}^{2}}}{c}m$, $k=\frac{(r-1)ab}{c}m$, $\mu=t=\frac{r{{a}^{2}}}{c}m$, $\lambda=\frac{(ar-b)a}{c}m$. Then $4k=n+2\lambda+2\mu$ implies that
$$4\frac{(r-1)ab}{c}m=\frac{(r-1){{b}^{2}}}{c}m+2\frac{r{{a}^{2}}}{c}m+2\frac{(ar-b)a}{c}m,$$
then $4(r-1)ab=(r-1){{b}^{2}}+2 r{{a}^{2}}+2(ar-b)a$. Thus
\[r(b-2a)^2=b(b-2a).\tag{6}\]
\textbf{Case 1}. If $b-2a=0$, then $c=\gcd (2a^2,r{{a}^{2}},4(r-1){a^{2}})=a^2\gcd(2,r)$ and $n=\frac{4 (r-1)}{(2,r)}m,k=\frac{2 (r-1)}{(2,r)}m,\mu=\frac{r}{(2,r)}m,\lambda=\frac{r-2}{(2,r)}m,t=\frac{r}{(2,r)}m$.\;If $2\nmid r$, then $(n,k,\mu,\lambda,t)=(4(r-1)m,2(r-1)m,rm,(r-2)m,rm)$.\;If $2|r$, then $(n,k,\mu,\lambda,t)=(2(r-1)m,(r-1)m,\frac{r}{2}m,(\frac{r}{2}-1)m,\frac{r}{2}m)$. \\
\textbf{Case 2}. If $b-2a\neq0$, then $r(b-2a)=b$, $b(r-1)=2ar$, $b\mid 2ar$. Since $gcd(a,b)=1$, we can obtain that $b|2r$.

If $2\nmid b$, then $b|r$. Hence $b-2a=1$, $r=b=2a+1$ and $c=\gcd(a(2a+1),(2a+1)a^2,2a(2a+1)^2)=a(2a+1)=\frac{r(r-1)}{2}$. Thus\; $(n,k,\mu,\lambda,t)=( \frac{(r-1){{b}^{2}}}{c}m,\frac{(r-1)ab}{c}m,$\\$\frac{r{{a}^{2}}}{c}m,\frac{(ar-b)a}{c}m,\frac{r{{a}^{2}}}{c}m )=(2rm,(r-1)m,\frac{r-1}{2}m,\frac{r-3}{2}m,\frac{r-1}{2}m)$.

If $2\mid b$, let $b=2b'$, then $b'|r$. By equation (6), we have $r(b'-a)=b'$. Thus $b'-a=1$, $b=2(a+1)$ and $r=a+1$. Then $c=\gcd(2a (a+1),(a+1)a^2,4a(a+1)^2)=a(a+1)gcd(2,a,4(a+1))=a(a+1)=r(r-1)$, so\;$(n,k,\mu,\lambda,t)=(4rm,2(r-1)m,(r-1)m,(r-3)m,(r-1)m)$.
\end{proof}
\begin{rmk}\label{r-1}
These possible parameters can be rewritten as\\
$(R1)$ $(2nm,nm,(\frac{n}{2}+1)m,(\frac{n}{2}-1)m,(\frac{n}{2}+1)m)$ for integer $4|n$ ($2(r-1)$ is replaced by $n$);\\
$(R2)$$(2nm,nm,\frac{n+1}{2}m,\frac{n-1}{2}m,\frac{n+1}{2}m)$ for odd integer $n$($r-1$ is replaced by $n$);\\
$(R3)$$(2nm,(n-1)m,\frac{n-1}{2}m,\frac{n-3}{2}m,\frac{n-1}{2}m)$ for odd integer $n$($r$ is replaced by $n$);\\
$(R4)$$(2nm,(n-2)m,(\frac{n}{2}-1)m,(\frac{n}{2}-3)m,(\frac{n}{2}-1)m)$ integer $4|n$ ($2r$ is replaced by $n$).
\end{rmk}

For odd $n$, the DSRG with parameters $(2n,n-1,\frac{n-1}{2},\frac{n-3}{2},\frac{n-1}{2})$ corresponding to case $(R3)$ in Remark \ref{r-1} is realizable which was constructed in \cite{K1} by using Cayley graph.\;Then from the Proposition \ref{p-3}, the complementary graph of $(2n,n-1,\frac{n-1}{2},\frac{n-3}{2},\frac{n-1}{2})-$DSRG is a DSRG with parameters $(2n,n,\frac{n+1}{2},\frac{n-1}{2},\frac{n+1}{2})$ which corresponding to case $(R2)$.

For even $n$, the DSRG with parameters $(2n,n-1,\frac{n}{2}-1,\frac{n}{2}-1,\frac{n}{2})$  was constructed in \cite{HO}. Then for integer $4|n$, the complementary graph of it is a DSRG with parameters $(2n,n,\frac{n}{2}+1,\frac{n}{2}-1,\frac{n}{2}+1)$ which corresponding to case $(R1)$. Thus we have following 3 Theorems.
\begin{thm}\label{t-4}
For odd integer $n$, there exist directed strongly regular graph with the parameters $(4n^2,2n^2-2,n^2-1,n^2-3,n^2-1)$.
\end{thm}
\begin{proof}
A DSRG with parameters$(2n,n-1,\frac{n-1}{2},\frac{n-3}{2},\frac{n-1}{2})$ was constructed, then Theorem \ref{t-4} follows from Theorem \ref{t-1}.
\end{proof}

\begin{thm}\label{t-5}
For odd integer $n$, there exist directed strongly regular graph with the parameters $(4n^2,2n^2,n^2+1,n^2-1,n^2+1)$.
\end{thm}
\begin{proof}
A DSRG with parameters$(2n,n,\frac{n+1}{2},\frac{n-1}{2},\frac{n+1}{2})$ was constructed, then Theorem \ref{t-5} follows from Theorem \ref{t-1}.
\end{proof}

\begin{thm}\label{t-6}
For integer $4|n$, there exist directed strongly regular graph with the parameters $(4n^2,2n^2,n^2+4,n^2-4,n^2+4)$.
\end{thm}
\begin{proof}
A DSRG with parameters$(2n,n,\frac{n}{2}+1,\frac{n}{2}-1, \frac{n}{2}+1)$ was constructed, then Theorem \ref{t-6} follows from Theorem \ref{t-1}.
\end{proof}
\section{Constructions of DSRG(Cayley graphs) by using Semidirect product}
In \cite{D} Duval, Art M., and Dmitri Iourinski constructed a new infinite family of directed strongly regular
graphs, as Cayley graphs of certain semidirect product groups, these results generalizes
an earlier construction of Klin,  Munemasa, Muzychuk, and Zieschang on some
dihedral groups.

In this section, using classical number theory, we also construct some new families directed strongly regular Cayley graphs $\mathcal{C}(G, S)$, where groups $G$ is semidirect product of two groups.

At first, we give some basic definitions of \emph{Cayley graph}, \emph{gruop ring}, Semidirect product and \emph{primitive root}. Let $e_G$ or simply $e$ to be the identity element of $G$. The following definitions are the same with Duval \cite{D}.
\begin{dfn}\label{d-2}
 Let $G$ be a finite group and $S\subset G\backslash e$(Remove identity element $e$ from $G$). The Cayley graph of $G$ generated
by $S$, denoted by $\mathcal{C}(G, S)$, is the digraph $H$ such that $V(H)=G$ and $x\rightarrow y$ if and only if
$x^{-1}y\in S$, for $\forall x, y\in G$.
\end{dfn}
\begin{dfn}\label{d-3}
For any finite group $G$, the group ring $\mathbb{Z}[G]$ is defined as the set of all formal
sums of elements of $G$, with coefficients from $\mathbb{Z}$. i.e.
\[\mathbb{Z}[G]=\left\{\sum_{g\in G}r_gg|r_g\in\mathbb{Z}, r_g\neq 0\text{ for finite g}\right\}.\]
The operations $+$ and $\cdot$ on $\mathbb{Z}[G]$ are given
by
\[\sum_{g\in G}r_gg+\sum_{g\in G}s_gg=\sum_{g\in G}(r_g+s_g)g,\]
\[\left(\sum_{g\in G}r_gg\right)\cdot\left(\sum_{g\in G}r_gg\right)=\left(\sum_{g\in G}t_gg\right),t_g=\sum_{g'g''=g}r_{g'}s_{g''}\]
\end{dfn}
 For any subset $X$ of $G$, Let $\overline{X}$ denote the
the element of the group ring $\mathbb{Z}[G]$ that is the sum of all elements of $X$. i.e. \[\overline{X}=\sum_{x\in X}x\]

The Lemma below allows us to express a sufficient and necessary condition for a Cayley graph to be directed strongly regular graph in terms of the group ring.

\begin{lem}\label{l-7}
The Cayley graph $\mathcal{C}(G, S)$ of $G$ with respect to $S$ is a directed strongly regular graph with parameters $(n, k, \mu,  \lambda, t)$ if and only if $|G| = n$, $|S|= k$, and
\[\overline{S}^2=te+\lambda\overline{S}+\mu(\overline{G}-e-\overline{S}).\tag{7}\]
\end{lem}

\begin{dfn}\label{d-4}
Let $\theta:B\rightarrow Aut\;A$ be an action of a group $B$ on another group A: Let
$A\rtimes_\theta B$ be the direct product set of $A$ and $B$; i.e., set of pairs $(a, b)$ operation for the product of two elements
\[(a,b)(a',b')=(a[\theta(b)(a')], bb')\]
Then $A\rtimes_\theta B$ is a group of order $|A||B|$, This group is called the semidirect product of $A$ and $B$
with respect to the action $\theta$.
\end{dfn}
From the definition of semidirect product of $A$ and $B$, we can obtain that $A$ is isomorphic to $A\times e_{B}$,\;i.e.,\;$A\cong A\times e_{B}$, and $B$ is isomorphic to $e_{A}\times B$,\;i.e.,\;$B\cong e_{A}\times B$.
 Then we can equate $a$ with $(a, e_B)$ for $a\in A$, and $b$ with $(e_A, b)$ for $b\in B$, and it is easy to verify that $ab=(a, e_B)(e_A, b)=(a, b)$, $ba=(e_A, b)(a, e_B)=(a, b)(a', b')=([\theta(b)(a)], b)=[\theta(b)(a)]b$, $e_A=(e_A, e_B)=e_B=e$($e$ is the identity element of $A\rtimes_\theta B$).
\begin{dfn}\label{d-5}
For positive integer $s$, which coprime with $n$, we denote $\delta_n(s)$ with the the least positive integer $k$ such that $s^k\equiv 1(mod\hspace{3pt}n)$. If $\delta_n(s)$ is equal to the totient of $n$
, the number of positive integers that are both less $n$ and coprime to $n$, i.e. $\delta_n(s)=\varphi(n)$, then we say that $s$ is a primitive root of modulo $n$.
\end{dfn}
From the classical number theory, the following result characterizes the existence of primitive root of modulo $n$.

\begin{lem}\label{l-8}
There is a primitive root of modulo $n$ if and only if $n=2$ or $4$, or $n=p^k$ or $2p^k$ for  odd prime an $k\in\mathbb{N}$.
\end{lem}
Let $C_{n}=\langle{x}\rangle$ be multiplicative cyclic groups of orders $n$ with generator $x$.

\subsection{Cayley graph $\mathcal{C}{(C_{p}\rtimes_{\theta(m)} C_{n}, S)}$ to be DSRG}
Let $A=C_{p}=\langle{a}\rangle$ be multiplicative cyclic groups of odd prime orders $p$, $B=C_{n}=\langle{x}\rangle$ be another multiplicative cyclic groups of  orders $n$.

Let $m$ be an integer such that $(m, p)=1$, $m\not\equiv1(mod \hspace*{3pt}p)$, $m^{n}\equiv1(mod \hspace*{3pt}p)$. The map $\beta_m\in{Aut\hspace*{3pt}C_{p}}$ given by $\beta_m$ :\;$a^i\rightarrow a^{mi}$\;is an automorphism, and the map $\theta(m)$: $B\rightarrow Aut\hspace*{3pt}C_{p}$ given by $\theta(m)(x^u)=\beta_m^{u}$\hspace*{3pt}is a homomorphism. Then $A\rtimes_{\theta(m)} B=\langle{a, x|a^{p}=x^n=e, xa=a^mx}\rangle$ \hspace*{3pt}is a group of order $pn$.

Let $H$ be a proper subset of $\{1, 2, \cdots, p-1\}$. We denote $v$ with the cardinalty of $H$, i.e.$|H|=v$.
\subsubsection{\texorpdfstring{$(pn, vn, \frac{n}{p-1}v^2, \frac{n}{p-1}v(v-1), \frac{n}{p-1}v^2)-$}{Lg}DSRG for \texorpdfstring{$1\leq v\leq p-2$}{Lg}}
\begin{thm}\label{t-9}
Let $A=C_{p}$, $p$ be an odd prime, $m$ be a primitive root of module $p$, $n$ be an integer such that $p-1\mid n$. Let $B=\langle x|x^{n}=e_{B}\rangle$ and $A'=\{a^l|l\in H\}$.\;Then the Cayley graph $\mathcal{C}{(A\rtimes_{\theta(m)} B, A'\times B)}$ is a directed strongly regular graph with the parameters
\[(pn, vn, \frac{n}{p-1}v^2, \frac{n}{p-1}v(v-1), \frac{n}{p-1}v^2).\]
\end{thm}

\begin{proof}
Let $S=A'\times B$, first we compute $\overline{B}\hspace*{3pt}\overline{A'}\hspace*{3pt}\overline{B}$\hspace*{3pt} in the group ring $\mathbb{Z}[A\rtimes_{\theta(m)} B]$:
\[\overline{B}\hspace*{3pt}\overline{A'}\hspace*{3pt}\overline{B}=\left(\sum_{u=0}^{n-1}x^{u} \right)\left(\sum_{l\in H}a^k\right)\overline{B}\hspace{40pt}\]
\[=\sum_{l\in H}\sum_{u=0}^{n-1}[\theta(x^{u})(a^l)]x^u\overline{B}\hspace*{7pt}\]
\[=\sum_{l\in H}\sum_{u=0}^{n-1}\beta_m^{u}(a^l)\overline{B}\hspace*{32pt}\]
\[=\sum_{l\in H}\sum_{u=0}^{n-1}a^{lm^{u}}\overline{B}\hspace{39pt}\]
\[=\frac{n|H|}{p-1}\sum_{l=1}^{p-1}a^{l}\overline{B}\hspace{39pt}.\tag{8}\]
The last equation follows from $gcd(l, p)=1$ for each $l\in H$.\;Then
\[\overline{S}^2=(\overline{A'\times B})^2=\overline{A'}\hspace{2pt}\overline{B}\hspace{2pt}\overline{A'}\hspace{2pt}\overline{B}\hspace{65pt}\]
\[=\frac{n|H|}{p-1}\sum_{l'\in H}\sum_{l=1}^{p-1}a^{l'+l}\overline{B}\hspace{59pt}\]
\[=\frac{n|H|}{p-1}\sum_{l'\in H}\sum_{\substack{l=0\\l\neq l'}}^{p-1}a^{l}\overline{B}\hspace{70pt}\]
\[=\frac{n|H|^2}{p-1}\sum_{l=0}^{p-1}a^{l}\overline{B}-\frac{n|H|}{p-1}\sum_{l'\in H}a^{l'}\overline{B}\hspace{9pt}\]
\[=\frac{nv^2}{p-1}\overline{A\rtimes_\theta B}-\frac{nv}{p-1}\overline{S}.\hspace{45pt}\tag{9}\]
Thus
\[\overline{S}^2=\frac{nv^2}{p-1}e+(\frac{nv^2}{p-1}-\frac{nv}{p-1})\overline{S}+\frac{nv^2}{p-1}(\overline{A\rtimes_\theta B}-e-\overline{S}).\]
Hence from Lemma \ref{l-7}, Cayley graph $\mathcal{C}{(A\rtimes_{\theta(m)} B, A'\times B)}$ is a directed strongly regular graph with the parameters
\[(pn, vn, \frac{n}{p-1}v^2, \frac{n}{p-1}v(v-1), \frac{n}{p-1}v^2).\]
\end{proof}
\begin{rmk}\label{r-3}
Indeed, if $v=p-1$, then we can also get $(\alpha p(p-1), \alpha (p-1)^2, \alpha (p-1)^2, \alpha (p-2)(p-1), \alpha (p-1)^2)-$DSRG, but it is a SRG in this case.
\end{rmk}

\subsubsection{\texorpdfstring{$(pn, n(v+1)-1, \frac{nv(v+1)}{p-1}, n-2+\frac{nv^2}{p-1}, n-1+\frac{nv^2}{p-1})-$}{Lg}DSRG for \texorpdfstring{$1\leq v\leq p-2$}{Lg}}

The subset $(A'+e_A)\times B\backslash e_A$ of $A\rtimes_{\theta(m)} B$ means that remove element $e_A$ from $(A'+e_A)\times B$, so we can get $(A'+e_A)\times B\backslash e_A=\{a^lx^j, a^c|1\leq j\leq n-1, l\in H\cup\{0\}, c\in H\}$,
\begin{thm}\label{t-10}
Let $A=C_{p}=\langle a\rangle$, $p$ be an odd prime, $m$ be a primitive root of module $p$, $n$ be an integer such that $p-1\mid n$. Let $B=C_n=\langle x|x^{n}=e_{B}\rangle$, and $A'=\{a^l|l\in H\}$.\;Then the Cayley graph $\mathcal{C}{(A\rtimes_{\theta(m)} B, (A'+e_A)\times B\backslash e_A)}$ is a directed strongly regular graph with the parameters
\[(pn, n(v+1)-1, \frac{nv(v+1)}{p-1}, n-2+\frac{nv^2}{p-1}, n-1+\frac{nv^2}{p-1}).\]
\end{thm}
\begin{proof}
Let $S=(A'+e_A)\times B\backslash e_A)$, then $\overline{S}=\sum\limits_{k\in H}a^k\overline{B}+\overline{B}- e_A=\overline{A'}\;\overline{B}+\overline{B}-e_A$. Thus
\[\overline{S}^2=\left(\overline{A'}\hspace*{3pt}\overline{B}+\overline{B}-e_A\right)^2\hspace{163pt}\]
\[\hspace{7pt}=(\overline{A'}\hspace*{3pt}\overline{B})^2+(\overline{B}-e_A)^2+(\overline{B}-e_A)\overline{A'}\hspace*{3pt}\overline{B}+\overline{A'}\hspace*{3pt}\overline{B}(\overline{B}-e_A)\]
\[\hspace{38pt}=(\overline{A'}\hspace*{3pt}\overline{B})^2+(|B|-2)\overline{B}+e_A+\overline{B}\;\overline{A'}\hspace*{3pt}\overline{B}-\overline{A'}\hspace*{3pt}\overline{B}+(|B|-1)\overline{A'}\hspace*{3pt}\overline{B}\]
\[=(\overline{A'}\hspace*{3pt}\overline{B})^2+e_A+\overline{B}\;\overline{A'}\hspace*{3pt}\overline{B}+(|B|-2)(\overline{A'}\hspace*{3pt}\overline{B}+\overline{B}).\hspace{26pt}\]
It follows from equations (8) and (9) in Theorem \ref{t-9} that
\[\overline{B}\;\overline{A'}\;\overline{B}=\frac{nv}{p-1}\sum_{k=1}^{p-1}a^{k}\overline{B}=\frac{nv}{p-1}(\overline{A}-e_A)\overline{B}\]
and \[(\overline{A'}\hspace*{3pt}\overline{B})^2=\frac{nv^2}{p-1}\overline{A\rtimes_{\theta(m)} B}-\frac{nv}{p-1}\overline{A'}\;\overline{B}.\]
Then $\overline{S}^2$ becomes
\[\hspace{8pt}=\frac{nv^2}{p-1}\overline{A\rtimes_{\theta(m)} B}-\frac{nv}{p-1}\overline{A'}\;\overline{B}+\frac{nv}{p-1}(\overline{A}-e_A)\overline{B}+e_A+(|B|-2)(\overline{S}+e_A)\]
\[\hspace{9pt}=\frac{nv^2}{p-1}\overline{A\rtimes_{\theta(m)} B}-\frac{nv}{p-1}(\overline{A'}+e_A)\overline{B}+\frac{nv}{p-1}\overline{A\rtimes_{\theta(m)} B}+(n-2)\overline{S}+(n-1)e_A\]
\[\hspace{2pt}=\frac{nv(v+1)}{p-1}\overline{A\rtimes_{\theta(m)} B}-\frac{nv}{p-1}(\overline{S}+e_A)+(n-2)\overline{S}+(n-1)e_A\hspace{50pt}\]
\[\hspace{2pt}=\frac{nv(v+1)}{p-1}\overline{A\rtimes_{\theta(m)} B}+(n-2-\frac{nv}{p-1})\overline{S}+(n-1-\frac{nv}{p-1})e.\hspace{53pt}\]
Thus
\[\overline{S}^2=(n-1+\frac{nv^2}{p-1})e+(n-2+\frac{nv^2}{p-1})\overline{S}+\frac{nv(v+1)}{p-1}(\overline{A\rtimes_{\theta(m)} B}-e-\overline{S}).\]
Hence from Lemma \ref{l-7}, Cayley graph $\mathcal{C}{(A\rtimes_{\theta(m)} B, S)}$ is a directed strongly regular graph with the parameters
\[(pn, n(v+1)-1, \frac{nv(v+1)}{p-1}, n-2+\frac{nv^2}{p-1}, n-1+\frac{nv^2}{p-1}).\]
\end{proof}
\subsection{\texorpdfstring{$(mq, m+q-2, \frac{m-1}{q}+1, \frac{m-1}{q}+q-2, \frac{m-1}{q}+q-1)-$}{Lg}DSRG}
For $A$ be a finite group of order $m$ and some $\beta \in Aut\;A$, Duval, Art M., and Dmitri Iourinski \cite{D} define a group automorphism $\beta$ which has the $q$-orbit condition if each of its untrivial orbits contains $q$ elements. i.e. If the group $\langle\beta\rangle$ (\;The subgroup of $Aut\;A$ generated by $\beta$) acts on $A$ naturally,\;then each untrivial orbit have $q$ elements. In other words, $\beta$ satisfies $\beta^{q}(a) = a$, for all $a \in A$, and $\beta^{u}(a) = a$ implies $q|u$, for all $a \neq e_{A}$.

 Let $A$ be a finite group of order $m$ and some $\beta \in Aut\;A$ has the $q$-orbit condition, let $B$ be the cyclic group of order $q$   with generator $b$, and define $\theta \colon B \rightarrow Aut\;A$ by $\theta(b^{u})=\beta^{u}$. let $A'$ be a set of representatives of the nontrivial orbits of $\beta$. Duval observed that the Cayley graph $\mathcal{C}(A \rtimes_{\theta} B,  A' \times B)$
is a directed strongly regular graph with the parameters $$(mq, \ m-1, \ (m-1)/q, \ ((m-1)/q)-1, \ (m-1)/q).$$

A similar result will be given in the next Theorem.
\begin{thm}\label{t-11}
Let $A$ be a finite group of order $m$.\;If some $\beta \in Aut\;A$ has the
$q$-orbit condition, let $B$ be the cyclic group of order $q$ with generator $b$, and define $\theta \colon B \rightarrow Aut\;A$ by $\theta(b^{u})=\beta^{u}$ for $1\leq u\leq q$. Let $A'$ be a set of representatives of the nontrivial orbits of $\beta$. Then the Cayley graph $\mathcal{C}(A \rtimes_{\theta} B,(A'+e_A)\times B\backslash e_A)$
    is a directed strongly regular graph with the parameters
    \[(mq, m+q-2, \frac{m-1}{q}+1, \frac{m-1}{q}+q-2, \frac{m-1}{q}+q-1)\]
\end{thm}
\begin{proof}
We recall that $\overline{B}\hspace*{3pt}\overline{A'}\hspace*{3pt}\overline{B}=(\overline{A}-e_A)\overline{B}=\overline{A}\hspace*{3pt}\overline{B}-\overline{B}$,
$\overline{A'}\hspace*{3pt}\overline{B}\hspace*{3pt}\overline{A'}\hspace*{3pt}\overline{B}=|A'|\overline{A\rtimes_\theta B}-\overline{A'}\hspace*{3pt}\overline{B}$ from the proof of Theorem 3.3 in \cite{D}. Let $S=(A'+e_A)\times B\backslash e_A$, then
\[\overline{S}^2=\left(\overline{A'}\hspace*{3pt}\overline{B}+\overline{B}-e_A\right)^2\hspace{163pt}\]
\[\hspace{7pt}=(\overline{A'}\hspace*{3pt}\overline{B})^2+(\overline{B}-e_A)^2+(\overline{B}-e_A)\overline{A'}\hspace*{3pt}\overline{B}+\overline{A'}\hspace*{3pt}\overline{B}(\overline{B}-e_A)\]
\[\hspace{8pt}=|A'|\overline{A\rtimes_\theta B}-\overline{A'}\hspace*{3pt}\overline{B}+(|B|-2)\overline{B}+e_A+(|B|-1)\overline{A'}\hspace*{3pt}\overline{B}\]
\[+\overline{A}\hspace*{3pt}\overline{B}-\overline{B}-\overline{A'}\hspace*{3pt}\overline{B}\hspace{140pt}\]
\[=(|A'|+1)\overline{A\rtimes_\theta B}+(|B|-3)\overline{B}+e_A+(|B|-3)\overline{A'}\hspace*{3pt}\overline{B}\]
\[=(|A'|+1)\overline{A\rtimes_\theta B}+(|B|-3)\overline{S}+(|B|-2)e.\hspace{39pt}\]
Thus
\[\overline{S}^2=(|A'|+1)(\overline{A\rtimes_\theta B}-e-\overline{S})+(|A'|+|B|-2)\overline{S}+(|A'|+|B|-1)e.\]
It now only remains to note that
$|B|=q$, $|A'|=\frac{m-1}{q}$, $|A\rtimes_\theta B|=mq$, $|(A'+e_A)\times B\backslash e_A|=(\frac{m-1}{q}+1)q-1=m+q-2$. Hence from Lemma \ref{l-7}, Cayley graph $\mathcal{C}{(A\rtimes_\theta B, (A'+e_A)\times B\backslash e_A)}$ is a DSRG with the parameters
    \[(mq, m+q-2, \frac{m-1}{q}+1, \frac{m-1}{q}+q-2, \frac{m-1}{q}+q-1)\]
\end{proof}

\subsection{\texorpdfstring{$(p^{2}n, p(p-1)n, n((p-1)^2+1), n\frac{(p-1)^3-1}{p-1}, n((p-1)^2+1))-$}{Lg}DSRG }
In the previous sections, we concentrate on the Cayley graph $\mathcal{C}{(G, S)}$, $G$ is a semidirect product of two cyclic group. In this section, we discuss Cayley graph $\mathcal{C}{(G, S)}$ with $G=(C_p\rtimes_{\theta(s)}C_n)\rtimes_\vartheta B$, where
$C_p=\langle a\rangle$, $C_n=\langle x\rangle$, and $B=\langle y|y^{p}=e_{B}\rangle$ is a  cyclic group of order $p$ with generator $y$.
Let $s$ be an integer such that $(s, p)=1$, $s\not\equiv1(mod \hspace*{3pt}p)$, $s^{n}\equiv1(mod \hspace*{3pt}p)$.\;The map $\beta_s\in{Aut\hspace*{3pt}\langle a\rangle}$ given by $\beta_s :a^i\rightarrow a^{si}$ is an automorphism, and the homomorphism $\theta(s)$ from $\langle x\rangle$ to $Aut\; \langle a\rangle$ is defined by: $x^\alpha\rightarrow \beta_s^\alpha$.

Let $D_{p, n, s}=C_p\rtimes_{\theta(s)}C_n$, the inner automorphism $f(g)\in{Aut\hspace*{3pt}D_{p, n, s}}$ is defined by $f(g)$: $\phi\rightarrow g^{-1}\phi g$, for each $\phi\in D_{p, n, s}$. And the map $\vartheta$: $B\rightarrow Aut\hspace*{3pt}D_{p, n, s}$ given by $\vartheta(y^u)=f(a^u)$\hspace*{3pt}is a homomorphism. Thus,
$D_{p, n, s}\rtimes_\vartheta B$\hspace*{3pt}is a group of order $p^2n$. From the definition, we can obtain $x^ta^u=a^{us^t}x^t$ easily.

Let $\mathcal{E}(n)$ denote with the set of positive integers that are both less n and coprime to n,\;i.e.$\mathcal{E}(n)=\{q|1\leq q\leq{n-1}, (q, n)=1\}$,\;so $|\mathcal{E}(n)|=\varphi(n)$.

The next two Lemmas will be used in the proof of Theorem \ref{t-14}.
\begin{lem}\label{l-12}
Let $A=C_{p^l}=\langle a\rangle$ and $A'=\{a^k|k\in\mathcal{E}(p^l)\}$.\;Then
\[\overline{A'}^2=(p^l-p^{l-1})\overline{A}-p^{l-1}\overline{A'}.\]

\end{lem}
\begin{proof}
We compute $\overline{A'}^2$ in the group ring $\mathbb{Z}[C_{p^l}]$:
\[\overline{A'}^2=\left(\overline{A}-\sum_{k=1}^{p^{l-1}}a^{pk}\right)^2\hspace*{47pt}\]
\[\hspace*{80pt}=\overline{A}^2-\sum_{k=1}^{p^{l-1}}a^{pk}\overline{A}-\overline{A}\sum_{k=1}^{p^{l-1}}a^{pk}+\left(\sum_{k=1}^{p^{l-1}}a^{pk}\right)^2\]
\[\hspace*{57pt}=(p^l-2p^{l-1})\overline{A}+\sum_{j=1}^{p^{l-1}}\left(\sum_{k=1}^{p^{l-1}}a^{p(k+j)}\right).\]

For each $1\leq j\leq p^{l-1}$, when $k$ takes over the complete residue system of module $p^{l-1}$, the $k+j$ also takes over the complete residue system of module $p^{l-1}$. Then we can get that $\sum\limits_{k=1}^{p^{l-1}}a^{p(k+j)}=\sum\limits_{k=1}^{p^{l-1}}a^{pk}$. Thus we can obtain that
\[\hspace*{25pt}\overline{A'}^2=(p^l-2p^{l-1})\overline{A}+p^{l-1}\left(\sum_{k=1}^{p^{l-1}}a^{pk}\right)\]
\[\hspace*{27pt}=(p^l-2p^{l-1})\overline{A}+p^{l-1}(\overline{A}-\overline{A'})\]
\[\hspace*{9pt}=(p^l-p^{l-1})\overline{A}-p^{l-1}\overline{A'}.\hspace*{11pt}\]
\end{proof}
\begin{lem}\label{l-13}
Let $H\subseteq\{0, 1, \cdots, p-1\}$, $T\subseteq\{0, 1, \cdots, n-1\}$, $S=\{a^lx^i|l\in H, i\in T\}\times B$. Then
\[\overline{S}^2=\sum_{u=0}^{m-1}\sum_{l', l\in H}\sum_{t', i\in T}a^{l'+(l-u+us^i)s^{i'}}x^{i'+i}\overline{B}.\]
\end{lem}
\begin{proof}
First we compute $\overline{B}\hspace*{3pt}\overline{A'}\hspace*{3pt}\overline{B}$\hspace*{3pt}in the group ring $\mathbb{Z}[D_{p, n, s}\rtimes_\vartheta B]$:
\[\overline{B}\hspace*{3pt}\overline{A'}\hspace*{3pt}\overline{B}=\left(\sum_{u=0}^{m-1} y^u\right)\left(\sum_{l\in H, i\in T}a^lx^i\right)\overline{B}\hspace{60pt}\]
\[=\sum_{l\in H, i\in T}\sum_{u=0}^{m-1} y^u(a^lx^i)\overline{B}\hspace{55pt}\]
\[=\sum_{l\in H, i\in T}\sum_{u=0}^{m-1} [\vartheta(y^u)(a^lx^i)]y^u\overline{B}\hspace{27pt}\]
\[=\sum_{l\in H, i\in T}\sum_{u=0}^{m-1} a^{l-u}x^ia^{u}\overline{B}\hspace{50pt}\]
\[=\sum_{l\in H, i\in T}\sum_{u=0}^{m-1} a^{l-u+us^i}x^i\overline{B}.\hspace{40pt}\]
Thus
\[\overline{S}^2=(\overline{A'\times B})^2=\overline{A'}\hspace{2pt}\overline{B}\hspace{2pt}\overline{A'}\hspace{2pt}\overline{B}\hspace{65pt}\]
\[=\overline{A'}\sum_{l\in H, i\in T}\sum_{u=0}^{m-1} a^{l-u+us^i}x^i\overline{B}\hspace{30pt}\]
\[=\sum_{u=0}^{m-1}\sum_{l', l\in H}\sum_{i', i\in T}a^{l'}x^{i'}a^{l-u+us^i}x^i\overline{B}\hspace{3pt}\]
\[\hspace{15pt}=\sum_{u=0}^{m-1}\sum_{l',l\in H}\sum_{i', i\in T}a^{l'+(l-u+us^i)s^{i'}}x^{i'+i}\overline{B}.\]
\end{proof}
\begin{thm}\label{t-14}
Let $A=D_{p, n, s}=C_p\rtimes_{\theta(s)}C_n$, $p$ be a prime, $s$ be a primitive root of module $p$, $n$ be an integer such that $p-1\mid n$. $B=\langle y|y^{p}=e_{B}\rangle$, $A'=\{a^lx^i|l\in\mathcal{E}(p), i\in\{0, 1\cdots, n-1\}\}$. Then the Cayley graph $\mathcal{C}{(A\rtimes_\vartheta B, A'\times B)}$ is a directed strongly regular graph with the parameters
\[(p^{2}n, p(p-1)n, n((p-1)^2+1), n\frac{(p-1)^3-1}{p-1}, n((p-1)^2+1)).\]
\end{thm}

\begin{proof}
Let $S=A'\times B$, then from Lemma \ref{l-13} and Lemma \ref{l-12} for $l=1$, we have
\[\overline{S}^2=\sum_{u=0}^{p-1}\sum_{l', l\in\mathcal{E}(p)}\sum_{i', i=0}^{n-1}a^{l'+(l-u+us^i)s^{i'}}x^{i'+i}\overline{B}\hspace{20pt}\]
\[\hspace{56pt}=\sum_{i', i=0}^{n-1}\sum_{l', l\in\mathcal{E}(p)}a^{l'+ls^{i'}}\sum_{u=0}^{p-1}a^{u(s^i-1)s^{i'}}x^{i'+i}\overline{B}\text{(Lemma \ref{l-12})}\]
\[\hspace{4pt}=\sum_{i', i=0}^{n-1}\overline{X}\sum_{u=0}^{p-1}a^{u(s^i-1)s^{i'}}x^{i'+i}\overline{B}\hspace{54pt}\]
\[=\sum_{\substack{i', i=0\\s^i\not\equiv1(mod p)}}^{n-1}\overline{X}\sum_{u=0}^{p-1}a^{u(s^i-1)s^{i'}}x^{i'+i}\overline{B}\hspace{28pt}\]
\[+p\sum_{\substack{i', i=0\\s^i\equiv1(mod p)}}^{n-1}\overline{X}x^{i'+i}\overline{B}\triangleq \Delta_1+\Delta_2,\hspace{32pt}\]
where $\overline{X}=(p-1)\sum\limits_{l=0}^{p-1}a^{l}-\sum\limits_{l\in\mathcal{E}(p)}a^{l}$. We note that $gcd((s^i-1)s^{i'}, p)=1$ for each $i$ satisfies $s^i\not\equiv1(mod\;p)$. Then
\[\Delta_1=\sum_{\substack{i', i=0\\s^i\not\equiv1(mod p)}}^{n-1}\left((p-1)\sum\limits_{l=0}^{p-1}a^{l}-\sum\limits_{l\in\mathcal{E}(p)}a^{l}\right)\sum_{u=0}^{p-1}a^{u(s^i-1)s^{i'}}a^{i'+i}\overline{B}\]
\[=\sum_{\substack{i', i=0\\s^i\not\equiv1(mod p)}}^{n-1}\left((p-1)\sum\limits_{l=0}^{p-1}a^{l}-\sum\limits_{l\in\mathcal{E}(p)}a^{l}\right)\left(\sum_{u=0}^{p-1}a^{u}\right)x^{i'+i}\overline{B}\]
\[=(p(p-1)-\varphi(p))\sum_{\substack{i', i=0\\s^i\not\equiv1(mod p)}}^{n-1}\left(\sum_{u=0}^{p-1}a^{u}\right)x^{i'+i}\overline{B}\hspace{47pt}\]
\[=v(p-1)^2\sum_{i'=0}^{n-1}\left(\sum_{u=0}^{p-1}a^{u}\right)x^{i'}\overline{B}=v(p-1)^2\overline{A\rtimes_\vartheta B},\hspace{30pt}\]
where $v=|\{i|s^i\not\equiv1(mod p), i\in\{0, 1, \cdots, n-1\}\}|=n-\frac{n}{p-1}$. And
\[\Delta_2=p\sum_{\substack{i', i=0\\s^i\equiv1(mod p)}}^{n-1}\left((p-1)\sum_{l=0}^{p-1}a^{l}-\sum_{l\in\mathcal{E}(p)}a^{l}\right)x^{i'+i}\overline{B}\hspace{52pt}\]
\[=p(n-v)\sum_{i'=0}^{n-1}\left((p-1)\sum_{l=0}^{p-1}a^{l}-\sum_{l\in\mathcal{E}(p)}a^{l}\right)x^{i'}\overline{B}\hspace{43pt}\]
\[=p(n-v)(p-1)\overline{A\rtimes_\vartheta B}-p(n-v)\overline{S}.\hspace{84pt}\]
Thus
\[\overline{S}^2=(v(p-1)^2+p(n-v)(p-1))\overline{A\rtimes_\vartheta B}-p(n-v)\overline{S}.\hspace{12pt}\]

We note that $v(p-1)^2+p(n-v)(p-1)=n(p-1)(p-2)+pn=n(p^2-2p+2)=n((p-1)^2+1)$, so
\[\overline{S}^2=(n((p-1)^2+1)-p(n-v))(\overline{S}+e)+n((p-1)^2+1)(\overline{A\rtimes_\vartheta B}-e-\overline{S}).\]

It now only remains to note that $|A\times_\vartheta B|=p^{2}n$, $|A'\times B|=p\varphi(p)n=p(p-1)n$ and $v(p-1)^2+p(n-v)(p-1)-p(n-v)=n\frac{(p-1)^3-1}{p-1}$. Thus Cayley graph $\mathcal{C}{(A\rtimes_\vartheta B, A'\times B)}$ is a directed strongly regular graph with the parameters
\[(p^{2}n,p(p-1)n, n((p-1)^2+1), n\frac{(p-1)^3-1}{p-1}, n((p-1)^2+1)).\]

\end{proof}
\section{Constructions of DSRG by using Cayley coset graph}
For a group G and a subgroup $H\leq G$,  denote by $[G:H]$ the set of left cosets of $H$ in
G,  that is
\[[G:H]=\{xH|x\in G\}.\]
$|G:H|$ is the index  of $H$ in $G$. For any subset $S\subset G$,  we may define a digraph on $[G : H]$ as follows:
\begin{dfn}\label{d-6}

Let $G$ be a group, $H$ a subgroup of $G$,  and $S$ a subset of $G$. define
the Cayley coset graph of $G$ with respect to $H$ and $S$ to be the directed graph with vertex set
$[G : H]$ and such that,  for any $xH, yH\in V$,
$xH$ is connected to $yH$ if and only if $x^{-1}y\in HSH$,
and denote the digraph by $\mathcal{C}(G,H,HSH)$.
\end{dfn}
We note that if $H=1$, the Cayley coset graph is a Cayley graph. In this section we just give a sufficient and necessary condition
for a Cayley coset graph to be DSRG in terms of the group ring.

\begin{lem}\label{l-15}
The number of paths of length 2 from $xH$ to $yH$ in $\mathcal{C}(G,H,HSH)$ equals the
coefficient of $x^{-1}y$ in $\frac{1}{|H|}\overline{HSH}^2$.
\end{lem}
\begin{proof}
The coefficient of $x^{-1}y$ in $\overline{HSH}^2$ is the number of ordered pairs $(x_1, x_2)\in HSH\times HSH$ such that $x_1x_2=x^{-1}y$.

Let $\mathcal{Q}$ be the set of all the ordered pair $(x_1, x_2)\in HSH\times HSH$ such that $x_1x_2=x^{-1}y$, $\mathcal{P}$ be the set of all paths of length 2 from $xH$ to $yH$. We define a map $\eta:\mathcal{Q} \rightarrow \mathcal{P}$ by:
\[(x_1, x_2)\rightarrow p(x,xx_1,y),\]
where $p(x,z,y)$ denote with the path $xH\rightarrow zH\rightarrow yH$ of length 2.

Let $\eta^{-1}(p)$ denote with the preimage of $p$.

At first, we prove that $\eta^{-1}(p(x,z,y))=\{(x^{-1}zh, h^{-1}z^{-1}y)|h\in H\}$.

 We can get $\{(x^{-1}zh, h^{-1}z^{-1}y)|h\in H\}\subset\eta^{-1}(p(x,z,y))$ easily.\;It now only remains to prove that $\eta^{-1}(p(x,z,y))\subset\{(x^{-1}zh, h^{-1}z^{-1}y)|h\in H\}$.\;Indeed,\;for each $(x_1, x_2)\in \eta^{-1}(xH\rightarrow zH\rightarrow yH)$, we can obtain that $xx_1H=zH$,
i.e.\;$x_1\in x^{-1}zH$, so $(x_1, x_2)\in\{(x^{-1}zh, h^{-1}z^{-1}y)|h\in H\}$.\;so $\{(x^{-1}zh, h^{-1}z^{-1}y)|h\in H\}=\eta^{-1}(p(x,z,y))$ and $|\eta^{-1}(p(x,z,y))|=|H|$.

Secondly, the map $\eta$ is a surjection obviously. Thus
\[|\mathcal{Q}|=\sum_{p(x,z,y)\in \mathcal{P}}|\eta^{-1}(p(x,z,y))|=|H||\mathcal{P}|\]
 Then the result is now immediate.
\end{proof}

\begin{thm}\label{t-16}
The Cayley graph $\mathcal{C}(G;H;HSH)$ is a directed strongly regular graph with parameters $(n, k, \mu,  \lambda, t)$ if and only if $|G:H|=n$,  $\frac{|HSH|}{|H|}= k$,  and
\[\frac{1}{|H|}\overline{HSH}^2=te+\lambda\overline{HSH}+\mu(\overline{G}-e-\overline{HSH}).\]
\end{thm}
\begin{proof}
Suppose\;$\frac{1}{|H|}\overline{HSH}^2=te+\lambda\overline{HSH}+\mu(\overline{G}-e-\overline{HSH})$,  By Lemma \ref{l-15} then, the number of
paths of length 2 from $xH$ to $yH$ is: $t$ if $x^{-1}y=e$; i.e. $xH=yH$; $\lambda$ if $x^{-1}y\in HSH$, i.e. $xH\rightarrow yH$;
and $\mu$ otherwise. Thus, we can get $\mathcal{C}(G,H,HSH)$ is a directed strongly regular graph. The reverse direction is a direct consequence of Lemma \ref{l-15}.
\end{proof}
\begin{rmk}
if $H=1$, we can get Lemma \ref{l-7} from Theorem \ref{t-16} easily.
\end{rmk}

\section{Sufficient and necessary condition of Cayley graph \texorpdfstring{$\mathcal{C}{(C_n\rtimes_{\theta(k)} C_m, S)}$}{Lg} to be DSRG}
In this section, we focus on the sufficient and necessary conditions of Cayley graphs $\mathcal{C}{(C_n\rtimes_{\theta(k)} C_m,A'\times C_m)}$ and $\mathcal{C}{(C_n\rtimes_{\theta(k)} C_m,(A'+e_A)\times C_m\backslash e_A)}$ to be DSRG, where $A'$ is a proper subet of $C_n\setminus e $. These sufficient and necessary conditions can be used to verify Cayley graphs which constructed in section 3.1, 3.2, and \cite{D},\cite{HO} are DSRGs.

Let $x$ be a generator for $C_n$ and $y$ a generator for $C_m$. Let $k$ be an integer such that $(k, n)=1$, $k\not\equiv1(mod \hspace*{3pt}n)$, $k^{m}\equiv1(mod \hspace*{3pt}n)$.\;Then the map $\beta_k\in{Aut\hspace*{3pt}C_n}$ given by $\beta_k$ :$x^i\rightarrow x^{ki}$ is an automorphism, and the map $\theta(k)$: $C_m\rightarrow Aut\hspace*{3pt}C_n$\hspace*{3pt} given by $\theta(k)(y^u)=\beta_k^{u}$ is a homomorphism. $A'=\{x^a|a\in H\}$,\;where $H$ is a proper subset of $\{1, 2, \cdots, n-1\}$ with $|H|=v$.

\subsection{A sufficient and necessary condition of Cayley graph \texorpdfstring{\\$\mathcal{C}{(C_n\rtimes_{\theta(k)} C_m, A'\times C_m)}$}{Lg} to be DSRG}
At first, we give some basic definitions of group representation theory.

Let $V$ be a vector space over the field $\mathbb{C}$ of complex numbers and let
$GL(V)$ be the group of isomorphisms of $V$ onto itself. Let $GL_{n}(\mathbb{C})$ be the general linear group consisting of all invertible matrices of order $n$. A representation of $G$ in $V$ is a homomorphism from the group $G$ into the group $GL(V)$. Let $\rho$ and $\rho'$ be two representations of the same group $G$ in vector spaces $V$ and $V'$,\;these representations are said to be isomorphic if there exists a linear isomorphism $\omega:V\rightarrow V'$,\;which satisfies the identity
\[\omega\circ \psi(g)=\varphi(g)\circ \omega\]
for all $g\in G$.

Let $\rho$ be a representations of  group $G$ in vector spaces $V$ with $dim_{\mathbb{C}}V=n$. Choose a basis for $V$,\;and let $T:V\rightarrow \mathbb{C}^n$ be the isomorphism taking coordinates with respect to this basis.\;Then setting $\varphi_g=T\rho_gT^{-1}$,\;for $g\in G$,\;yields a representation $\varphi:G\rightarrow GL_{n}(\mathbb{C})$ isomorphic to $\rho$.\;So,\;lf $\rho$ and $\rho'$ are given in matrix form by $\varphi_g$ and $\varphi_g'$ respectively, then $\rho$ and $\rho'$ are isomorphic
means that there exists an invertible matrix $H$ such that for all $g\in G$
\[H\varphi_g=\varphi_g'H \text{\hspace{4pt}or\hspace{4pt}}\varphi_g=H^{-1}\varphi_g'H.\]

\begin{dfn}{(see \cite{FO})}\label{d-7}
Given a group $G$ and an element $g\in G$, let $A_g = A(\mathcal{C}(G, \{g\}))$,\;the adjacent matrix of Cayley graph $\mathcal{C}(G,\{g\})$.
Additionally,\;given a group $G$ and a multiset $S$ of elements of $G$,\;let $A_S = A(\mathcal{C}(G, S))$. Then we can define the map $\psi:G\rightarrow GL_{|G|}(\mathbb{C})$ gived by $g\rightarrow A_g$.
\end{dfn}
\begin{dfn}{(see \cite{FO})}\label{d-8}
Let $R_h$ be $h\times h$ matrix with entries $1$ in positions $(1, 2), (2, 3)$,$\cdots, (h-1, h), (h, 1)$.\;Let $\Omega_h=diag\{e^{2\pi i\frac{hk^{0}}{n}}, e^{2\pi i\frac{hk^{1}}{n}}, \cdots, e^{2\pi i\frac{hk^{m-1}}{n}}\}$ be $m\times m$ matrix satisfy $k^m\equiv 1(mod\hspace{3pt}n)$. Let $X=diag\{\Omega_1, \Omega_2, \cdots, \Omega_{n}\}$ and $Y=diag\{R_m, R_m,\cdots, R_m\}$ be two $n\times n$ block matrices with $m\times m$ blocks.\;
We can also define the map $\overline{\psi}:C_n\rtimes_{\theta(k)} C_m\rightarrow GL_{nm}(\mathbb{C})$ gived by $x^ay^b\rightarrow X^aY^b$.
\end{dfn}

\begin{lem}{(see \cite{FO})}\label{l-17}
$\psi$ and $\overline{\psi}$ are two representations of $C_n\rtimes_{\theta(k)} C_m$ in $GL_{nm}(\mathbb{C})$ and these two representations
of $C_n\rtimes_{\theta(k)} C_m$ are isomorphic group representations.
\end{lem}
\begin{rmk}{(see \cite{FO})}\label{r-4}
Now we know that these representations are isomorphic,\;then there exists an invertible matrix $H$ such that for each $g\in G$
\[\psi_g=H^{-1}\overline{\psi}_gH.\]
So $\psi_g$ and $\overline{\psi}_g$ will have the same characteristic polynomial and minimum  polynomial.
\end{rmk}
\begin{lem}{(see \cite{FO})}\label{l-18}
Given a group G and a multiset S of elements of G, then
\[A_S=\sum_{s\in S}A_s.\]
\end{lem}

The symbol $\chi(A, \gamma)$ and $\chi_0(A, \gamma)$ are denoted with characteristic polynomial and minimum polynomial of $A$ respectively.
\begin{lem}\label{l-19}
Let matrix $A=diag\{a_1,a_2,\cdots,a_n\}J_n$,\;then the minimum polynomial $\chi_0(A, \gamma)$ of $A$ is a factor of $\gamma(\gamma-\sum\limits_{i=1}^na_i)$.\;Particularly, if $\sum\limits_{i=1}^na_i\neq 0$, then $\chi_0(A, \gamma)=\gamma(\gamma-\sum\limits_{i=1}^na_i)$.
\end{lem}
\begin{proof}Let $s=\sum\limits_{i=1}^na_i$, then
\[A(A-sI)=\left(\begin{array}{cccc}
a_1 & & & \\
&a_2 & & \\
& & \ddots& \\
& & &a_n \\
\end{array}\right)J_n\left(\left(\begin{array}{cccc}
a_1 & & & \\
&a_2 & & \\
& & \ddots& \\
& & &a_n \\
\end{array}\right)J_n-sI\right)=0.\]
Thus $\chi_0(A, \gamma)|\gamma(\gamma-\sum\limits_{i=1}^na_i)$. Suppose $\sum\limits_{i=1}^na_i\neq 0$, then $A\neq 0$ and $A\neq \left(\sum\limits_{i=1}^na_i\right)I$, so $\chi_0(A, \gamma)=\gamma(\gamma-\sum\limits_{i=1}^na_i)$.
\end{proof}
\begin{dfn}\label{d-9}
Let $S_u=\sum\limits_{h=0}^{m-1}\sum\limits_{a\in H}e^{2\pi i\frac{uak^{h}}{n}}$ and $E_u(h)=\sum\limits_{a\in H}e^{2\pi i\frac{uak^{h}}{n}}$ for $0\leq u \leq n-1$, $0\leq h \leq m-1$.\;Then $S_0=vm$.
\end{dfn}
In \cite{FO}, Nathan Foxa gave an expression of characteristic polynomial of the semidirect product of
two cyclic groups and this result will be represented in the following Lemma.
\begin{lem}{(see \cite{FO})}\label{l-20}
The characteristic polynomial of the  semidirect product of two cyclic groups is given by the following:
\[\chi(A(\mathcal{C}(C_n\rtimes_{\theta(k)} C_m, S)), \gamma)=\prod_{u=0}^{n-1}\chi(\sum_{x^ay^b\in S}{(\Omega_{ua})}(R_m)^b, \gamma).\]
Particularly, let $S=A'\times C_m$, then
\[\chi(A(\mathcal{C}(C_n\rtimes_{\theta(k)} C_m, A'\times C_m)), \gamma)=\gamma^{n(m-1)}(\gamma-vm)\prod_{u=1}^{n-1}(\gamma-S_u).\]
\end{lem}
\begin{proof}
The first assertion is from \cite{FO}. To prove the second assertion, we note that if $S=A'\times C_m$, then
\[\sum_{x^ay^b\in S}{(\Omega_u)}^a(R_m)^b=\sum_{a\in H}{(\Omega_{ua})}\sum_{b=0}^{m-1}(R_m)^b=\sum_{a\in H}{(\Omega_{ua})}J\]
\[=\sum_{a\in H}diag\{e^{2\pi i\frac{uak^{0}}{n}}, e^{2\pi i\frac{uak^{1}}{n}}, \cdots, e^{2\pi i\frac{uak^{m-1}}{n}}\}J\hspace{15pt}\]
\[=diag\{E_u(0), E_u(1), \cdots, E_u(m-1)\}J.\hspace{51pt}\tag{10}\]
Then
\[\chi(\sum_{x^ay^b\in S}{(\Omega_{ua})}(R_m)^b, \gamma)=|\gamma I-\sum_{a\in H}{(\Omega_{ua})}J|=|\gamma I-\sum_{a\in H}{(\Omega_{ua})}\mathbf{1}_{m}\mathbf{1}_{m}^T|\]
\[=\gamma^{m-1}(\gamma-\mathbf{1}_{m}^T\sum_{a\in H}{(\Omega_{ua})}\mathbf{1}_{m})=\gamma^{m-1}(\gamma-S_u).\hspace{98pt}\]
Thus
\[\chi(A(\mathcal{C}(C_n\rtimes_{\theta(k)} C_m, S)), \gamma)=\prod_{u=0}^{n-1}\chi(\sum_{x^ay^b\in S}{(\Omega_{ua})}(R_m)^b, \gamma)=\prod_{u=0}^{n-1}\gamma^{m-1}(\gamma-S_u).\]
Then the result follows.
\end{proof}
In a similar way, we will give an expression of minimum  polynomial of the semidirect product of two cyclic groups.

Let $\mathbf{lcm}\{f_1(\gamma),f_2(\gamma),\cdots,f_{n-1}(\gamma),f_n(\gamma)\}=\mathbf{lcm}\{f_u(\gamma)|1\leq u\leq n\}$ denote with the least common multiple polynomial among the following polynomials  $f_1(\gamma),f_2(\gamma),\cdots,f_{n-1}(\gamma),f_n(\gamma)$.
\begin{lem}\label{l-21}
The minimum  polynomial of the semidirect product of two cyclic groups is given by the following:
\[\chi_0(A(\mathcal{C}(C_n\rtimes_{\theta(k)} C_m, A'\times C_m)), \gamma)\Big|\mathbf{lcm}\{\gamma(\gamma-S_u)|0\leq u\leq n-1\}.\]
Additionally, if $S_u\neq 0$ for each $1\leq u\leq n-1$, then
\[\chi_0(A(\mathcal{C}(C_n\rtimes_{\theta(k)} C_m, A'\times C_m)), \gamma)=\mathbf{lcm}\{\gamma(\gamma-S_u)|0\leq u\leq n-1\}.\]
\end{lem}
\begin{proof}
Let $G=C_n\rtimes_{\theta(k)} C_m$, $S=A'\times C_m$,\;Then from Lemma \ref{l-18}
\[\chi_0(A(\mathcal{C}(G, S)), \gamma)=\chi_0(\sum_{s\in S}A(\mathcal{C}(G, \{x\})), \gamma)=\chi_0(\sum_{s\in S}A_s, \gamma).\]
Since $s\in G$, it can be written uniquely as $x^ay^b$ for some $0\leq a<n$ and
some $0\leq b<m$. From the Lemma \ref{l-17} and Remark \ref{r-4}, we have
\[\chi_0(\sum_{x^ay^b\in S}A_{x^ay^b}, \gamma)=\chi_0(\sum_{x^ay^b\in S}X^aY^b, \gamma).\]
From the definition of $X$ and $Y$ in Definition \ref{d-8}, we can obtain that the minimum  polynomial of $A(\mathcal{C}(C_n\rtimes_{\theta(k)} C_m, A'\times C_m)$ is
\[\mathbf{lcm}\{\chi_0(\sum\limits_{x^ay^b\in S}{(\Omega_u)}^a(R_m)^b, \gamma)|0\leq u\leq n-1\}.\tag{11}\]
We can get
\[\chi_0(\sum_{x^ay^b\in S}{(\Omega_u)}^a(R_m)^b, \gamma)\Big|\gamma(\gamma-\sum_{h=0}^{m-1}\sum\limits_{a\in H}e^{2\pi i\frac{uak^{h}}{n}})=\gamma(\gamma-S_u)\]
from Lemma \ref{l-19} and the equation (10) in Lemma \ref{l-20}.\;Thus from $(11)$, we can obtain that
\[\chi_0(A(\mathcal{C}(C_n\rtimes_{\theta(k)} C_m, A'\times C_m)), \gamma)\Big|\mathbf{lcm}\{\gamma(\gamma-S_u)|0\leq u\leq n-1\}.\]
This proves the first assertion.\;Additionally, if $S_u\neq 0$ for each $1\leq u\leq n-1$, then from Lemma \ref{l-19},\;we have
\[\chi_0(A(\mathcal{C}(C_n\rtimes_{\theta(k)} C_m, A'\times C_m)), \gamma)=\mathbf{lcm}\{\gamma(\gamma-S_u)|0\leq u\leq n-1\}.\]
The second assertion is proved.
\end{proof}

In the following theroem, we will give a sufficient and necessary condition of Cayley graph $\mathcal{C}{(C_n\rtimes_{\theta(k)} C_m, A'\times C_m)}$ to be DSRG with parameters $(nm, vm, \mu, \lambda, t)$ and propositions $(\star)$ $t=\mu$, $\frac{vm}{\mu-\lambda}=n-1$ in term of $S_1$, $S_2$, $\cdots$, $S_{n-1}$.

We note that both $(pn, vn, \frac{n}{p-1}v^2, \frac{n}{p-1}v(v-1), \frac{n}{p-1}v^2)-$DSRG and $(mq, \ m-1, \ (m-1)/q, \ ((m-1)/q)-1, \ (m-1)/q)-$DSRG satisfy propositions $(\star)$.

The $r, s, \rho, \sigma$ occuring in the following Theorems \ref{t-22},\ref{t-23} are defined as Proposition \ref{p-2}.
\begin{thm}\label{t-22}
Cayley graph $\mathcal{C}{(C_n\rtimes_{\theta(k)} C_m, A'\times C_m)}$ is a DSRG with parameters $(nm, vm, \mu, \lambda, t)$ and $t=\mu$, $\frac{vm}{\mu-\lambda}=n-1$ if and only if $S_1=S_2\cdots=S_{n-1}$ are negative integers.
\end{thm}
\begin{proof}
Let $A=A(\mathcal{C}{(C_n\rtimes_{\theta(k)} C_m, A'\times C_m)})$. If $\mathcal{C}{(C_n\rtimes_{\theta(k)} C_m, A'\times C_m)}$ is a DSRG with parameters  $(nm, vm, \mu, \lambda, t)$ and $t=\mu$, $\frac{vm}{\mu-\lambda}=n-1$. Then from Proposition \ref{p-2},\;the three distinct eigenvalues of it are $vm$, $\rho=0$ and $\sigma=\lambda-\mu<0$ with multiplicities $1, r=\frac{-vm+(\mu-\lambda)\left(nm-1\right)}{\mu-\lambda}=n(m-1), s=\frac{vm+\rho \left( nm-1 \right)}{\rho -\sigma }=\frac{vm}{\mu-\lambda}=n-1$ respectively. From the Lemma \ref{l-20}, the characteristic polynomial of $A$ is $(\gamma-vm)\gamma^{n(m-1)}(\gamma-\sigma)^{n-1}=\gamma^{n(m-1)}(\gamma-vm)\prod\limits_{u=1}^{n-1}(\gamma-S_u)$,
so $S_1=S_2\cdots=S_{n-1}=\sigma<0$. \\
{\hspace*{20pt}}On the other hand, suppose $S_1=S_2\cdots=S_{n-1}$ are negative integer $\overline{d}$, then from the Lemma \ref{l-21},
the minimum polynomial of $A$ is $\mathbf{lcm}\{\gamma(\gamma-S_u)|0\leq u\leq n-1\}=\gamma(\gamma-vm)(\gamma-\overline{d})$, and
$AJ=JA=vm$.\;Then $A(A-vmI)(A-\overline{d}I)=0$.\;Let $B=A(A-\overline{d}I)$, so $(A-vmI)B=AB-vmB=0$, i.e. $AB=vmB$, then each column of $B$
is an eigenvector corresponding to simple eigenvalue $vm$ (from the Perron-Frobenius theory, can see, e.g., Horn and
Johnson\cite{RA}), but the eigenspace associated with the eigenvalue $vm$ has
dimension one and hence each column of $B$ is a suitable multiple of $\mathbf{1}_{nm}$. Let $B=(b_1\mathbf{1}_{nm}, b_2\mathbf{1}_{nm}, \cdots, b_{nm}\mathbf{1}_{nm})$, since $\mathbf{1}_{nm}^TB=\mathbf{1}_{nm}^TA(A-\overline{d}I)=vm(vm-\overline{d})\mathbf{1}_{nm}^T$, we can get ${nm}b_1={nm}b_2=\cdots={nm}b_n=vm(vm-e)$. Thus $A(A-\overline{d}I)=B=\frac{v(vm-\overline{d})}{n}J\triangleq\mu J$, then
\[A^2=\frac{v(vm-\overline{d})}{n}J+\overline{d}A=\mu I+\lambda A+\mu(J-I-A).\]
where $\mu+\overline{d}\triangleq\lambda$. Thus Cayley graph $\mathcal{C}{(C_n\rtimes_{\theta(k)} C_m, A'\times C_m)}$ is a DSRG with parameters $(nm, vm, \mu, \lambda, t)$ and $t=\mu$. Additionally, from Lemma \ref{l-20} the characteristic polynomial of $A$ is
$\chi(A(\mathcal{C}(C_n\rtimes_{\theta(k)} C_m, A'\times C_m)), \gamma)=\gamma^{n(m-1)}(\gamma-vm)\prod\limits_{u=1}^{n-1}(\gamma-\overline{d})=\gamma^{n(m-1)}(\gamma-vm)\prod\limits_{u=1}^{n-1}(\gamma-(\lambda-\mu))$, hence $\rho=0$,\;$\sigma=\lambda-\mu$, and the multiplicity of eigenvalue $\sigma=\lambda-\mu$ is $s=n-1$ which implies that $s=\frac{vm+\rho \left( nm-1 \right)}{\rho -\sigma }=\frac{vm}{\mu-\lambda}=n-1$,\;then the result follows.
\end{proof}
\begin{rmk}
We can verify Theorem \ref{t-9} from Theorem \ref{t-22}, recalling the conditions in the Theorem \ref{t-9}, Cayley graph $\mathcal{C}{(C_p\rtimes_{\theta(m)} C_n, A'\times C_n)}$ satisfies $S_0=vn$, and $$S_u=\sum\limits_{h=0}^{n-1}\sum\limits_{l\in H}e^{2\pi i\frac{ulm^{h}}{p}}=\sum\limits_{h=0}^{n-1}\sum\limits_{l\in H}e^{2\pi i\frac{ulm^{h}}{p}}=\frac{n}{p-1}\sum\limits_{h=1}^{p-1}\sum\limits_{l\in H}e^{2\pi i\frac{h}{p}}=-\frac{nv}{p-1}$$ for $1\leq u\leq p-1$, so $\mathcal{C}{(C_p\rtimes_{\theta(m)} C_n, A'\times C_n)}$ is a DSRG with parameters $(pn, vn, \mu, \lambda, t)$ such that $t=\mu$, $\lambda-\mu=-\frac{nv}{p-1}$, and $\mu=\frac{v(vn+\frac{nv}{p-1})}{p}=\frac{nv^2}{p-1}$, i.e.$(pn, vn, \frac{n}{p-1}v^2, \frac{n}{p-1}v(v-1), \frac{n}{p-1}v^2)$.
\end{rmk}

A generalized result will be exhibited in the following Theorem.
\begin{thm}\label{t-23}
Cayley graph $\mathcal{C}{(C_n\rtimes_{\theta(k)} C_m, A'\times C_m)}$ is a DSRG with given $s, \sigma<0$ if and only if $s$ numbers in $S_1, S_2\cdots, S_{n-1}$ have same value $\sigma$ and others are $0$, and if $S_u=0$, then $E_u(h)=0$ for $0\leq h\leq m-1$.
\end{thm}
\begin{proof}
Let $A=A(\mathcal{C}{(C_n\rtimes_{\theta(k)} C_m, A'\times C_m)})$. Suppose  $\mathcal{C}{(C_n\rtimes_{\theta(k)} C_m, A'\times C_m)}$ is a DSRG with given $s, \sigma<0$, then from Proposition \ref{p-2} and Lemma \ref{l-20}, the characteristic polynomial of $A$ is $(\gamma-vm)(\gamma-\rho)^{r}(\gamma-\sigma)^{s}=\gamma^{n(m-1)}(\gamma-vm)\prod\limits_{u=1}^{n-1}(\gamma-S_u)$.
Thus $\rho=0$, and $s$ numbers in $S_1, S_2\cdots, S_{n-1}$ have same value $\sigma$.\;And if $S_u=0$, we can obtain that $\sum\limits_{a\in H}{(\Omega_{ua})}J=0$.\;Otherwise, $\chi_0(\sum\limits_{a\in H}{(\Omega_{ua})}J, \gamma)=\gamma^2$, then from lemma \ref{l-21},\;the power of factor $\gamma$ in $\chi_0(A)$ greater than 1, this is a contradiction to $\chi_0(A)=(\gamma-vm)\gamma(\gamma-\sigma)$. Thus from equation (10),\;we can get $E_u(h)=0$ for each $0\leq h\leq m-1$. \\
{\hspace*{20pt}}On the other hand, if $S_1, S_2\cdots, S_{n-1}$ satisfy the conditions, then from the Lemma \ref{l-21},
the minimum polynomial of $A$ is $\mathbf{lcm}\{\gamma(\gamma-vm),\gamma,\gamma-\sigma\}=(\gamma-vm)\gamma(\gamma-\sigma)$.\;Similar to the proof of Theorem \ref{t-22},\;we can also obtain that the Cayley graph $\mathcal{C}{(C_n\rtimes_{\theta(k)} C_m, A'\times C_m)}$ is a DSRG.
Additionally,\;from Lemma \ref{l-20}, the characteristic polynomial of $A$ is
$\chi(A(\mathcal{C}(C_n\rtimes_{\theta(k)} C_m, A'\times C_m)), \gamma)=(\gamma-\sigma)^{s}(\gamma-vm)\gamma^r$, then the result follows.
\end{proof}
\begin{rmk}
If Cayley graph $\mathcal{C}{(C_n\rtimes_{\theta(k)} C_m, A'\times C_m)}$ is a DSRG with parameters $(nm, vm, \mu,  \lambda, t)$, then $t=\mu$.
\end{rmk}
\subsection{A sufficient and necessary condition of Cayley graph
\texorpdfstring{\\$\mathcal{C}{(C_n\rtimes_\theta C_m, (A'+e_A)\times C_m\backslash e_A)}$}{Lg} to be DSRG}
When $n$ is even,
 Hobart, Sylvia A.,\;and T.Justin Shaw in \cite{HO} constructed Cayley graph $\mathcal{C}(G, \widehat{S})$ with dihedral group $G=D_n=\langle b,a|b^n=a^2=e,ab=b^{-1}a\rangle$,\;and
$\widehat{S}=\{b, b^2, \cdots, b^{\frac{n}{2}-1}, a, ab, ab^2, \cdots, ab^{\frac{n}{2}-1}\}=\{b^0,b^1,\cdots,b^{\frac{n}{2}-1}\}\times\{a^0,a^1\}\backslash\{b^0 \}$.\;This Cayley graph is a DSRG with parameters $(2n, n-1, \frac{n}{2}-1, \frac{n}{2}-1, \frac{n}{2})$.

In this section, we will investigate $\mathcal{C}{(C_n\rtimes_\theta C_m, (A'+e_A)\times C_m\backslash e_A)}$ respect to it. Let $H$  be a proper subset of $\{1, 2, \cdots, n-1\}$  with $|H|=v$. And $(A'+e_A)\times C_m\backslash e_A=\{x^ay^b, x^c|1\leq b\leq m-1, a\in H\cup\{0\}, c\in H\}$.
\begin{dfn}\label{d-10}
Let $S_u^*=\sum\limits_{h=0}^{m-1}\sum\limits_{a\in H\cup\{0\}}e^{2\pi i\frac{uak^{h}}{n}}$ and $E_u^*(h)=\sum\limits_{a\in H\cup\{0\}}e^{2\pi i\frac{uak^{h}}{n}}$ for $0\leq u \leq n-1$, $0\leq h \leq m-1$. Then $S_0^*=(v+1)m$.
\end{dfn}
We can also obtain the characteristic polynomial and minimum polynomial of Cayley graph $\mathcal{C}{(C_n\rtimes_{\theta(k)} C_m, (A'+e_A)\times C_m\backslash e_A))}$ as section 5.1.\;Let $S^*=(A'+e_A)\times C_m\backslash e_A$.

\begin{lem}\label{l-24}
The characteristic polynomial of the $\mathcal{C}{(C_n\rtimes_{\theta(k)} C_m, S^*)}$ is given by the following:
\[\chi(A(\mathcal{C}(C_n\rtimes_{\theta(k)} C_m, S^*)), \gamma)=(\gamma+1)^{n(m-1)}(\gamma+1-(v+1)m)\prod_{u=1}^{n-1}(\gamma+1-S_u^*).\]
\end{lem}
\begin{proof}
We can obtain
\[\sum_{x^ay^b\in S^*}{(\Omega_u)}^a(R_m)^b=\sum_{a\in H\cup\{0\}}{(\Omega_{ua})}\sum_{b=0}^{m-1}(R_m)^b-\Omega_0=\sum_{a\in H\cup\{0\}}{(\Omega_{ua})}J-I\]
\[=\sum_{a\in H\cup\{0\}}diag\{e^{2\pi i\frac{uak^{0}}{n}}, e^{2\pi i\frac{uak^{1}}{n}}, \cdots, e^{2\pi i\frac{uak^{m-1}}{n}}\}J-I\]
\[=diag\{E_u^*(0), E_u^*(1), \cdots, E_u^*(m-1)\}J-I.\hspace{54pt}\tag{12}\]
and
\[\chi(\sum_{x^ay^b\in S^*}{(\Omega_{ua})}(R_m)^b, \gamma)=|(\gamma+1)I-\sum_{a\in H\cup\{0\}}{(\Omega_{ua})}\mathbf{1}_{m}\mathbf{1}_{m}^T|=(\gamma+1)^{m-1}(\gamma+1-S_u^*).\]
From the Lemma \ref{l-20}, we have
\[\chi(A(\mathcal{C}(C_n\rtimes_{\theta(k)} C_m, S^*),\gamma)=\prod_{u=0}^{n-1}(\gamma+1)^{m-1}(\gamma+1-S_u^*),\]
then the result follows.
\end{proof}

\begin{lem}\label{l-25}
The minimum polynomial of $\mathcal{C}{(C_n\rtimes_{\theta(k)} C_m, S^*)}$ is given by the following:
\[\chi_0(A(\mathcal{C}(C_n\rtimes_{\theta(k)} C_m, S^*)), \gamma)\Big|\mathbf{lcm}\{(\gamma+1)(\gamma+1-S_u^*)|0\leq u\leq n-1\}.\]
\end{lem}
\begin{proof}
From the Lemma \ref{l-19} and the equation (12) in Lemma \ref{l-24},\;we have
\[\chi_0(\sum_{x^ay^b\in S^*}{(\Omega_{ua})}(R_m)^b, \gamma)=\chi_0(\sum_{a\in H\cup\{0\}}{(\Omega_{ua})}J-
I, \gamma)=\chi_0(\sum_{a\in H\cup\{0\}}{(\Omega_{ua})}J, \gamma+1),\]
then the result follows from Lemma \ref{l-21} by replacing $\gamma$ with $\gamma+1$.
\end{proof}

\begin{thm}\label{t-26}
Cayley graph $\mathcal{C}{(C_n\rtimes_{\theta(k)} C_m, S^*)}$ is a DSRG with given $r, \rho$ if and only if $r$ numbers in $S_1^*, S_2^*\cdots, S_{n-1}^*$ have same value $1+\rho$ and others are $0$, and if $S_u^*=0$, then $E_u^*(h)=0$ for $0\leq h\leq m-1$.
\end{thm}
\begin{proof}
Let $A=A(\mathcal{C}{(C_n\rtimes_{\theta(k)} C_m, S^*)})$. Suppose Cayley graph $\mathcal{C}{(C_n\rtimes_{\theta(k)} C_m, S^*)}$ is a DSRG with given $r, \rho$. Then from Proposition \ref{p-2} and Lemma \ref{l-24},\;the characteristic polynomial of $A$ is $(\gamma+1-(v+1)m)(\gamma-\rho)^{r}(\gamma-\sigma)^{s}=\prod\limits_{u=0}^{n-1}(\gamma+1)^{m-1}(\gamma+1-S_u)$.
Thus $\sigma=-1$, and $r$ numbers in $S_1^*, S_2^*\cdots, S_{n-1}^*$ has same value $1+\rho$. We note that $\chi_0(A)=(\gamma+1-vm)(\gamma+1)(\gamma-\rho)$, which imply that if $S_u^*=0$, then $\sum\limits_{a\in H\cup\{0\}}{(\Omega_{ua})}J=0$. Thus from equation (12),\;we can get $E_u^*(h)=0$ for all $0\leq h\leq m-1$.\\
{\hspace*{20pt}}On the other hand, suppose $S_1^*, S_2^*\cdots, S_{n-1}^*$ satisfy the conditions, then from the Lemma \ref{l-25},
the minimum polynomial of $A$ is $\mathbf{lcm}\{\gamma(\gamma+1-(v+1)m),\gamma+1,\gamma-\rho\}=(\gamma+1-(v+1)m)(\gamma+1)(\gamma-\rho)$. Similar to the proof of Theorem \ref{t-22}, the Cayley graph $\mathcal{C}{(C_n\rtimes_{\theta(k)} C_m, S^*)}$ is a DSRG.
Additionally, the characteristic polynomial of $A$ is
$\chi(A(\mathcal{C}(C_n\rtimes_{\theta(k)} C_m, S^*)), \gamma)=(\gamma+1)^{s}(\gamma+1-(v+1)m)(\gamma-\rho)^r$, then the result follows.
\end{proof}
\begin{rmk}

(1).\;If Cayley graph $\mathcal{C}{(C_n\rtimes_{\theta(k)} C_m, S^*)}$ is a DSRG  with parameters $(nm, m(v+1)-1, \mu,  \lambda, t)$, then $t=\lambda+1$. Indeed,\;from Proposition \ref{p-1},\;$\sigma=-1$ imples that $(\lambda-\mu+2)^2=d=(\lambda-\mu)^2+4(t-\mu)$, i.e. $t=\lambda+1$. \\
(2).\;Using Theorem \ref{t-26},\;we can verify $\mathcal{C}(D_n, \widehat{S})$ constructed by Hobart, Sylvia A.,and T.Justin Shaw is a DSRG with parameters $(2n, n-1, \frac{n}{2}-1, \frac{n}{2}-1, \frac{n}{2})$ easily.\;Indeed,
 $\mathcal{C}(D_n, \widehat{S})$ satisfies $S_0=2\times\frac{n}{2}=n$, and $S_u^*=\sum\limits_{h=0}^{1}\sum\limits_{a=0}^{\frac{n}{2}-1}e^{2\pi i\frac{ua(-1)^{h}}{n}}=\sum\limits_{a=0}^{n-1}e^{2\pi i\frac{ua}{n}}+1-e^{\pi iu }=1-(-1)^u$. Thus
  \[S_u^*=\left\{\begin{array}{c}\
\setlength{\baselineskip}{30pt}
2,\;\text{if}\;2|u\;\;\;\\
0,\;\text{if}\; 2\nmid u
\end{array}\right.\]
  and if $S_u^*=0$, then $2|u$, so $E_u^*(h)=\sum\limits_{a=0}^{\frac{n}{2}-1}e^{2\pi i\frac{u(-1)^{h}}{n}}=\frac{1-e^{2\pi i\frac{u\frac{n}{2}(-1)^{h}}{n}}}{1-e^{2\pi i\frac{u(-1)^{h}}{n}}}=0$. From Theorem \ref{t-26}, the Cayley graph $\mathcal{C}(D_n, S)$ is a DSRG with $r=\frac{n}{2}$, $\rho=1$, $\sigma=-1$. Thus the parameters of $\mathcal{C}(D_n, S)$ is $(2n, n-1, \frac{n}{2}-1, \frac{n}{2}-1, \frac{n}{2})$.\\
(3).\;We can verify Theorem \ref{t-10} from Theorem \ref{t-26}. Indeed, under the hypothesis of Theorem \ref{t-10}, the Cayley graph $\mathcal{C}{(C_p\rtimes_{\theta(m)} C_n, (A'+e_A)\times C_m\backslash e_A)}$ satisfies $$S_u^*=\sum\limits_{h=0}^{n-1}\sum\limits_{l\in H\cup\{0\}}e^{2\pi i\frac{ulm^{h}}{p}}=\frac{n}{p-1}\sum\limits_{h=1}^{p-1}\sum\limits_{l\in H}e^{2\pi i\frac{h}{p}}+n=n-\frac{nv}{p-1}$$ for each $1\leq u\leq p-1$, so Cayley graph $\mathcal{C}{(C_p\rtimes_{\theta(m)} C_n, (A'+e_A)\times C_m\backslash e_A)}$ is a DSRG with $\rho=n-1-\frac{nv}{p-1}$, $r=n-1$, and $\sigma=-1$. Thus
the parameters of it is $(pn, n(v+1)-1, \frac{nv(v+1)}{p-1}, n-2+\frac{nv^2}{p-1}, n-1+\frac{nv^2}{p-1})$.
\end{rmk}
\section{The out(in)-neighbour set and automorphsim group of DSRG (Cayley graphs)}
\subsection{The out(in)-neighbour set of DSRG(Cayley graphs)}
In this section, we make a discussion of the vertices which have the same out-neighbour set (or in-neighbour set), the following Lemma gives an upper bound of the number of these vertices. Throughout this section, digraph $D$ is directed strongly regular graph  with parameters $(n, k, \mu, \lambda, t)$.
\begin{lem}{(see \cite{JG})}\label{l-27}
If a DSRG with parameters $( n, k, \mu , \lambda , t )$ has a set $S$ of vertices,\;all of which have the same set $N$ of out-neighbour set (or they all have the same set of in-neighbour set),  then $\left| S \right|\le k-\lambda$,  $\left| S \right|\le n-2k+t$.
\end{lem}
\begin{dfn}\label{d-11}
Let $S_1$, $S_2$, $\cdots$, $S_t$ be the partition of $V(D)$ such that any two vertices from the same $S_i$ have the same out-neighbour set and any two vertices from distinct $S_i$ have distinct out-neighbour set. We denote this partition with $P_{out}(D)=\{S_1, S_2, \cdots, S_t\}$. For in-neighbour set,\;we can also define $P_{in}(D)$ as above.
\end{dfn}
From Lemma \ref{l-27},\;we can obtain that $|S_i|\leq\min\{k-\lambda,n-2k+t\}$ for all $0\leq i\leq t$.\;The following Lemma gives an improvement of Lemma \ref{l-27}, and this Lemma gives an upper bound of $|S_i|+|S_j|$ for distinct $i$ and $j$.
\begin{lem}\label{l-28}
If $D$ a directed strongly regular graph with parameters $(n, k, \mu, \lambda, t)$, and $\lambda>0$.\;Then for distinct $0\leq i,j\leq t$,\;$|S_i|+|S_j|\leq max\{k, n-2k+2\beta\}$, where $\beta=\min\{k-\lambda, \mu, n-2k+t\}$.
\end{lem}
\begin{proof}
We can assume $i=1$ and $j=2$.\;At first, we claim that $N_{D}^{+}(S_2)\nsubseteq S_1$, $N_{D}^{+}(S_1)\nsubseteq S_2$, as $|S_1|\leq k-\lambda < k=|N_{D}^{+}(S_2)|$, $|S_2|\leq k-\lambda < k=|N_{D}^{+}(S_1)|$ from Lemma \ref{l-27}.\;Since $D$ is loopless,\;we have $S_1\cap S_2=\emptyset$, $S_1\cap N_{D}^{+}(S_1)=S_2\cap N_{D}^{+}(S_2)=\emptyset$ \\
\textbf{Case 1}. $|S_1|+|S_2|\leq k$. \\
\textbf{Case 2}. If $|S_1|+|S_2|>k$, then $N_{D}^{+}(S_1)\cap N_{D}^{+}(S_2)=\emptyset$. Thus $|S_1|+|S_2|\leq n+|S_1\cap N_{D}^{+}(S_2)|+|S_2\cap N_{D}^{+}(S_1)|-2k$.\;Let $u=|S_2\cap N_{D}^{+}(S_1)|\leq |S_2|\leq n-2k+t$, $z=|S_1\cap N_{D}^{+}(S_2)|\leq |S_1|\leq n-2k+t$, then $|S_1|+|S_2|\leq n-2k+u+z$. \\
\textbf{Case 2.1}. $S_1\nsubseteq N_{D}^{+}(S_2)$, $S_2\nsubseteq N_{D}^{+}(S_1)$. See the Figure 1.

Let $v\in S_1\backslash N_{D}^{+}(S_2)$, $w\in N_{D}^{+}(S_2)\backslash S_1$, $x\in S_2\backslash N_{D}^{+}(S_1)$, $y\in N_{D}^{+}(S_1)\backslash S_2$. Since $ v\nrightarrow w$, we can obtain that $S_2\cap N_{D}^{+}(S_1)\subseteq N_{D}^{+}(v)\cap N_{D}^{-}(w)$, so $u\leq \mu$; Since $ v\nrightarrow x$, then $|N_{D}^{+}(S_1)\backslash S_2|\geq \mu$, and $u\leq k-\mu$; Since $ v\rightarrow y$, then $|N_{D}^{+}(S_1)\backslash S_2|\geq \lambda$, and $u\leq k-\lambda$. Thus $u\leq min\{k-\mu, k-\lambda, \mu, n-2k+t\}\triangleq \alpha$.\;In a similar way,\;$z\leq\alpha$,\;then $|S_1|+|S_2|\leq n-2k+u+z=n-2k+2\alpha$.
\begin{figure}
  \centering
  \includegraphics[width=8cm]{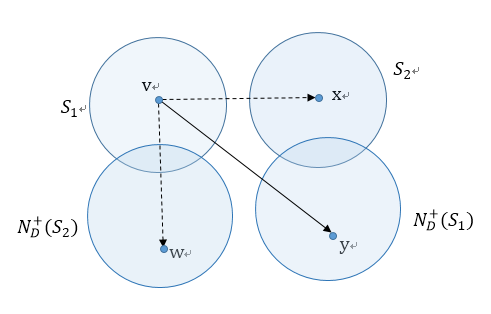}\\
   \caption{$S_1\nsubseteq N_{D}^{+}(S_2)$, $S_2\nsubseteq N_{D}^{+}(S_1)$}
\end{figure}
\\
\textbf{Case 2.2}. $S_1\nsubseteq N_{D}^{+}(S_2)$, $S_2\subseteq N_{D}^{+}(S_1)$. See the Figure 2.

We can also get $z\leq min\{k-\mu, k-\lambda, \mu, n-2k+t\}\doteq \alpha$ as \textbf{case 2.1}. To give an upper bound on $u$, let $v\in S_1\backslash N_{D}^{+}(S_2)$, $w\in N_{D}^{+}(S_2)\backslash S_1$, and $y\in N_{D}^{+}(S_1)\backslash S_2$.\;Similar to the discussion of \textbf{case 2.1},\;we can also obtain that $u\leq min\{k-\lambda, \mu, n-2k+t\}\triangleq \beta$.\;Thus $|S_1|+|S_2|\leq n-2k+u+z=n-2k+\alpha+\beta\leq n-2k+2\beta$.
\\
\textbf{Case 2.3}. $S_2\nsubseteq N_{D}^{+}(S_1)$, $S_1\subseteq N_{D}^{+}(S_2)$. We also have $|S_1|+|S_2|\leq n-2k+u+z=n-2k+\alpha+\beta\leq n-2k+2\beta$.
\\
\textbf{Case 2.4}. $S_2\subseteq N_{D}^{+}(S_1)$, $S_1\subseteq N_{D}^{+}(S_2)$. We also have $|S_1|+|S_2|\leq n-2k+2\beta$.

Base on our discussion,\;the result follows.
\begin{figure}[H]
  \centering
  \includegraphics[width=8cm]{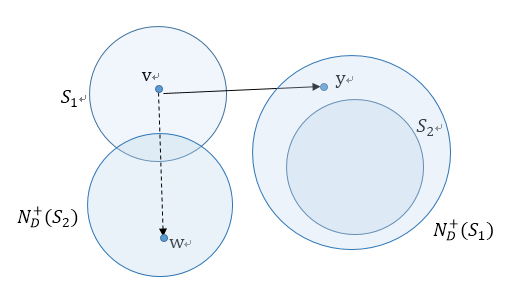}\\
   \caption{$S_1\nsubseteq N_{D}^{+}(S_2)$, $S_2\subseteq N_{D}^{+}(S_1)$}
\end{figure}
\end{proof}

\begin{thm}
Let $D$ be a directed strongly regular Cayley graph $\mathcal{C}(G, S)$ with parameters $( n, k, \mu , \lambda , t )$. Let $G_S=\{g|g\in G, gS=S\}$ be a subgroup of $G$, and $G=\bigcup\limits_{i=1}^{q}a_iG_S$ to be the left coset decomposition of $G$ with respect
to $G_S$, where $q=|G:G_S|$ is the index  of $G_S$ in $G$.\;Then we have\\
$(a)$ $P_{out}(D)=\{a_1G_S, a_2G_S, \cdots, a_qG_S\}$;\\
$(b)$ $|a_iG_S|=|G_S|\leq min\{k-\lambda,  n-2k+t\}$;\\
$(c)$ $|G_S|$ is a factor of $gcd(n, k, \lambda, \mu, t)$;\\
$(d)$ $G_S=\{1_G\}$ if $t\neq \mu$ or $gcd(n, k, \lambda, \mu, t)=1$.
\end{thm}
\begin{proof}
We note that for distinct vertices $x, y\in G$, the out-neighbour set of $x, y$ are $xS, yS$ respectively, then $xS=yS$ $\Leftrightarrow$ $x^{-1}y\in G_{S}$ $\Leftrightarrow$ $xG_{S}=yG_{S}$. Hence $P_{out}(D)=\{a_1G_{S}, a_2G_{S}, \cdots, a_qG_{S}\}$, proving $(a)$. The assertion $(b)$ follows from Lemma \ref{l-27} and the assertion $(c)$ can be directly deduced by the Definition \ref{d-11}. Finally, if $t\neq \mu$, then for each $g\in G_{S}$, we multiply $g$ on the left side of equation $(7)$, we can get
\[\overline{S}^2=tg+\lambda\overline{S}+\mu(\overline{G}-g-\overline{S})=(t-\mu)g+(\lambda-\mu)\overline{S}+\mu\overline{G}.\]
Note that $\overline{S}^2=(t-\mu)e+(\lambda-\mu)\overline{S}+\mu\overline{G}$, so $g=e$ and $G_S=\{e_G\}$. On the other hand, if $gcd(n, k, \lambda, \mu, t)=1$, then from assertion $(c)$, we can also obtain that $|G_S|=1$, $G_S=\{e_G\}$, proving $(d)$.
\end{proof}

Let $S^{-1}=\{a^{-1}|a\in S\}$, then we can also obtain the following theorem as above.
\begin{thm}
Let $D$ be a directed strongly regular Cayley graph $\mathcal{C}(G, S)$ with parameters $( n, k, \mu , \lambda , t )$.\;Let $G_{S^{-1}}=\{g|g\in G, gS^{-1}=S^{-1}\}$, and $G=\bigcup\limits_{i=1}^{q'}a_iG_{S^{-1}}$ to be the left coset decomposition of $G$ with respect
to $G_{S^{-1}}$, where $q'=|G:G_{S^{-1}}|$ is the index  of $G_{S^{-1}}$ in $G$.\;Then we have\\
$(a)$ $P_{in}(D)=\{a_1'G_{S^{-1}}, a_2'G_{S^{-1}}, \cdots, a_{q'}'G_{S^{-1}}\}$;\\
$(b)$ $|a_iG_{S^{-1}}|=|G_{S^{-1}}|\leq min\{k-\lambda,  n-2k+t\}$;\\
$(c)$ $|G_{S^{-1}}|$ is a factor of $gcd(n, k, \lambda, \mu, t)$;\\
$(d)$ $G_{S^{-1}}=\{1_G\}$ if $gcd(n, k, \lambda, \mu, t)=1$.
\end{thm}
\begin{proof}
We note that for distinct vertices $x, y\in G$, the in-neighbour set of $x, y$ are $xS^{-1}, yS^{-1}$ respectively, then $xS^{-1}=yS^{-1}$ $\Leftrightarrow$ $x^{-1}y\in G_{S^{-1}}$ $\Leftrightarrow$ $xG_{S^{-1}}=yG_{S^{-1}}$. Then $P_{in}(D)=\{a_1'G_{S^{-1}}, a_2'G_{S^{-1}}, \cdots, a_t'G_{S^{-1}}\}$, proving $(a)$. The assertion $(b)$ follows from Lemma \ref{l-27}, and the assertion $(c)$ can be directly deduced by the Definition \ref{d-11}. Finally, from assertion $(c)$, we can also obtain that $|G_S|=1$, $G_S=\{e_G\}$, proving $(d)$.
\end{proof}
\begin{dfn}{\label{d-12}}
Let $\mathcal{C}(G, S)$ be a DSRG.\;The digraph $\mathcal{S}_{out}(\mathcal{C}(G, S))$ is defined by $V(\mathcal{S}_{out}(\mathcal{C}(G, S)))=[G:G_S]=\{a_1G_{S}, a_2G_{S}, \cdots, a_qG_{S}\}$, and $a_iG_{S}\rightarrow a_jG_{S}$ if only if $a_i^{-1}a_j\in S$, i.e.\;$E(\mathcal{S}_{out}(\mathcal{C}(G, S)))=\{(a_iG_{S}, a_jG_{S})|(a_i, a_j)\in\mathcal{C}(G, S)\}.$

We can also define digraph $\mathcal{S}_{in}(\mathcal{C}(G, S))$ with $V(\mathcal{S}_{in}(\mathcal{C}(G, S)))=[G:G_{S^{-1}}]=\{a_1'G_{S^{-1}}, a_2'G_{S^{-1}}, \cdots, a_{q'}'G_{S^{-1}}\}$ and $E(\mathcal{S}_{in}(\mathcal{C}(G, S)))=\{(a_iG_{S^{-1}}, a_jG_{S^{-1}})|\\(a_i, a_j)\in\mathcal{C}(G, S)\}$.
\end{dfn}
\begin{thm}
Let $\mathcal{C}(G, S)$ be a directed strongly regular Cayley graph with parameters $( n, k, \mu , \lambda , t )$,
then digraphs\\
(1)\;$\mathcal{S}_{out}(\mathcal{C}(G, S))$ is a $(\frac{n}{|G_S|}, \frac{k}{|G_S|}, \frac{\mu}{|G_S|}, \frac{\lambda}{|G_S|}, \frac{t}{|G_S|})-$DSRG;\\
(2)\;$\mathcal{S}_{in}(\mathcal{C}(G, S))$ is a $(\frac{n}{|G_{S^{-1}}|}, \frac{k}{|G_{S^{-1}}|},\frac{\mu}{|G_{S^{-1}}|},\frac{\lambda}{|G_{S^{-1}}|}, \frac{t}{|G_{S^{-1}}|})-$DSRG.
\end{thm}
\begin{proof}
We first prove the assertion for $\mathcal{S}_{out}(\mathcal{C}(G, S))$.
From the definition of $\mathcal{S}_{out}(\mathcal{C}(G, S))$,\;we have $|V(\mathcal{S}_{out}(\mathcal{C}(G, S)))|$$=\frac{n}{|G_S|}$, $d_{\mathcal{S}_{out}(\mathcal{C}(G, S))}^{-}(a_iG_S)=\frac{k}{|G_S|}$,\;and\; $d_{\mathcal{S}_{out}(\mathcal{C}(G, S))}^{+}(a_iG_S)\geq\frac{k}{|G_S|}$ for all $1\leq i\leq q$.
We note that
\[\sum_{i=1}^td_{\mathcal{S}_{out}(\mathcal{C}(G, S))}^{-}(a_iG_S)=\sum_{i=1}^td_{\mathcal{S}_{out}(\mathcal{C}(G, S))}^{+}(a_iG_S).\]
Then $d_{\mathcal{S}_{out}(\mathcal{C}(G, S))}^{-}(a_iG_S)=d_{\mathcal{S}_{out}(\mathcal{C}(G, S))}^{+}(a_iG_S)=\frac{k}{|G_S|}$, so $\mathcal{S}_{out}(\mathcal{C}(G, S))$ is regular digraph.

 If $a_iG_S \rightarrow a_jG_S$, then $a_i\rightarrow a_j$ in $\mathcal{C}(G, S)$. Hence in $\mathcal{C}(G, S)$,\;the number of paths of length two from the vertex $a_i$ to the vertex $a_j$ is $\lambda$,\;so the number of paths of length two from the vertex $a_iG_S$ to the vertex $a_jG_S$ is $\frac{\lambda}{|G_S|}$ in $\mathcal{S}_{out}(\mathcal{C}(G, S))$.

  If $a_iG_S \nrightarrow a_jG_S$, then $a_i\nrightarrow a_j$ in $\mathcal{C}(G, S)$. Hence in $\mathcal{C}(G, S)$,
  the number of paths of length two from the vertex $a_i$ to the vertex $a_j$ is $\mu$, then the number of paths of length two from the vertex $a_iG_S$ to the vertex $a_jG_S$ is $\frac{\mu}{|G_S|}$ in $\mathcal{S}_{out}(\mathcal{C}(G, S))$.\; If $a_iG_S=a_jG_S$, then $a_iG_S$ contains in the $\frac{t}{|G_S|}$ $2-cycle$ exactly. Thus $\mathcal{S}_{out}(\mathcal{C}(G, S))$ is a DSRG with parameters $(\frac{n}{|G_S|}, \frac{k}{|G_S|}, \frac{\mu}{|G_S|}, \frac{\lambda}{|G_S|}, \frac{t}{|G_S|})$.\;
   The similar result for $\mathcal{S}_{in}(\mathcal{C}(G, S))$ is also hold, since $d_{\mathcal{S}_{in}(\mathcal{C}(G, S))}^{+}(a_iG_{S^{-1}})=\frac{k}{|G_{S^{-1}}|}$, $d_{\mathcal{S}_{in}(\mathcal{C}(G, S))}^{-}(a_iG_{S^{-1}})\geq\frac{k}{|G_{S^{-1}}|}$ for each $1\leq i\leq q'$.
\end{proof}
\begin{rmk}These digraphs $\mathcal{S}_{out}(\mathcal{C}(G, S))$ and  $\mathcal{S}_{in}(\mathcal{C}(G, S))$ are different from Cayley coset graph.
\end{rmk}
\subsection{The automorphsim group of DSRG(Cayley graphs)}
In this section,\;we give an upper bound of $|Aut(\mathcal{C}(G, S))|$ in term of $|G_S|$ and $|G_{S^{-1}}|$.
\begin{thm}
Let $D$ be a directed strongly regular Cayley graph $\mathcal{C}(G, S)$ with parameters $( n, k, \mu , \lambda , t )$.\;Then $$|Aut(\mathcal{C}(G, S))|\leq \min\{(\frac{n}{|G_S|})!(|G_S|)!,(\frac{n}{|G_{S^{-1}}|})!(|G_{S^{-1}}|)!\}.$$
\end{thm}
\begin{proof}
let $\tau\in Aut(\mathcal{C}(G, S))$, if $x$ and $y$ have the same out-neighbour set, then $\tau(x)$ and $\tau(y)$ also have the same out(in)-neighbour set. Thus for each $1\leq i\leq q$, $\tau$ maps all vertices in $a_iG_S$ to another coset $a_{j_{i}}G_S$ for some $j_i$, where $j_1, j_2, \cdots, j_q$ is a permutation of ${1, 2, \cdots, q}$. Thus
$|Aut(\mathcal{C}(G, S))|\leq(\frac{n}{|G_S|})!(|G_S|)!$. In a similar way,\;we can also get
$|Aut(\mathcal{C}(G, S))|\leq(\frac{n}{|G_S^{-1}|})!(|G_{S^{-1}}|)!$.
\end{proof}
\section*{Acknowledgements}
The author are grateful to those of you who support to us.
\section*{References}

\bibliographystyle{plain}

\bibliography{666666}
\end{document}